\documentclass{article}
\usepackage{graphicx} 
\usepackage{amsthm,amsfonts, amsbsy,amssymb,amsmath,graphicx}
\usepackage{graphics}
\usepackage[english]{babel}
\usepackage{graphicx}
\usepackage{lineno}
\usepackage[utf8]{inputenc}
\usepackage{enumerate,enumitem}
\usepackage{tkz-berge}
\usepackage{hyperref}
\usepackage{color}
\usepackage{anyfontsize}
\usepackage{lmodern}
\usepackage{cite} 
\usepackage{float}
\usepackage{tikz}          
\usetikzlibrary{positioning,shapes}

\usepackage{xcolor}        
\usetikzlibrary{graphs}    
\usetikzlibrary{positioning}
\usepackage{tkz-graph}
\usetikzlibrary{shapes.geometric}

\hypersetup{colorlinks=true}

\hypersetup{colorlinks=true, linkcolor=blue, citecolor=blue,urlcolor=blue}

\input epsf

\usepackage{epic,latexsym,mathtools}

\usepackage{epsf}

\usepackage{amsfonts}
\usepackage{amscd}
\usepackage{amsmath}
\usepackage{graphicx}
\usepackage{color}
\usepackage{caption,subcaption}

\newtheorem{theorem}{Theorem}[section]

\newtheorem{proposition}[theorem]{Proposition}

\newtheorem{problem}{Problem}

\newtheorem{corollary}[theorem]{Corollary}

\tikzstyle{vertex}=[circle, draw, inner sep=0pt, minimum size=7pt]

\title{Injective (edge) colorings of generalized Sierpi\'{n}ski graphs}
\author{ckbhanupriya@gmail.com; bostjan.bresar@um.si}
\author{C.~K. Bhanupriya$^a$ and Bo\v{s}tjan Bre\v{s}ar$^{b,c}$\footnote{Corresponding author}, \\ \\
$^a$ Department of Mathematics, Madanapalle Institute of Technology and Science, \\Deemed to be University, Andhra Pradesh, India \\
\\
$^b$ Faculty of Natural Sciences and Mathematics, University of Maribor, Slovenia \\
\\
$^c$ Institute of Mathematics, Physics and Mechanics, Ljubljana, Slovenia\\
\\
\small \tt Email: ckbhanupriya@gmail.com\\
\small \tt Email: bostjan.bresar@um.si \\
}

\date{}

\begin{document}

\maketitle

\begin{abstract}
Generalized Sierpi\'{n}ski graphs constitute a distinctive class of fractal-like networks with recursive definition: given a graph $G$, $S_G^1=G$ while $S_G^n$ is obtained from $|V(G)|$ copies of $S_G^{n-1}$ by adding some edges in a prescribed way that reflects the structure of $G$. Many graph invariants have been studied in generalized Sierpi\'{n}ski graphs, and in this paper we focus on their injective colorings, both the vertex and the edge version. Given a graph $G$, a mapping $f$ that assigns an integer from $\{1,\ldots,k\}$ to each vertex (resp.\ edge) of $G$ is an injective (edge) coloring of $G$ if $f(x)=f(y)$ implies that $x$ and $y$ are not in a common triangle nor at distance $2$ for any two vertices (resp.\ edges) $x$ and $y$ in $G$.
The minimum number of colors $k$ for which there exists an injective (edge) coloring of $G$ is called the injective chromatic number (resp.\ injective chromatic index) of $G$ and is denoted by $\chi_i(G)$ (resp.\ $\chi_i'(G)$). The vertex version of injective colorings in generalized Sierpi\'{n}ski graphs was studied in [Injective colorings of Sierpi\'nski-like graphs and Kneser graphs,
Graphs.\ Combin.\ 41 (2025) 83], where the authors determined the injective chromatic numbers of the (standard) Sierpi\'{n}ski graphs,
and asked about the values when $G$ is a cycle. We resolve this question by proving that $\chi_i(S_{C_k}^n)=3$ for every $n\ge 2$ and every $k\ge 3$. Moreover, we prove an almost conclusive result that $\chi_i(S_G^n)\in \{\chi_i(G),\chi_i(G)+1\}$  for any graph $G$ and any $n\ge 2$. Injective edge colorings appear to be more difficult, especially in graphs with (many) triangles. On a positive note, we prove that $\chi_i'(S_{K_3}^n)=5$ for all $n\ge 3$, while $\chi_i'(S_{K_3}^2)=4$ and $\chi_i'(S_{K_3}^1)=3$. Furthermore, if $G$ is a triangle-free graph, we prove that $\chi_i'(S_G^n)\in \{\chi_i'(S_G^3),\chi_i'(S_G^3)+1\}$ for all $n\ge 4$, and provide some sufficient conditions on an injective edge coloring of the 3-dimensional Sierpi\'{n}ski graph over $G$, which ensure that $\chi_i'(S_G^n)=\chi_i'(S_G^3)$. In particular, the latter result enables us to establish that $\chi_i'(S_{C_4}^n)=3$, $\chi_i'(S_{C_5}^n)=4$ and $\chi_i'(S_{C_6}^n)=3$ hold for all $n\ge 2$.     
\end{abstract}


\noindent\textbf{AMS Subj.\ Class.\ (2020)}: 05C15, 05C76    

\section{Introduction}

Given a graph $G$, a function $f:V(G)\rightarrow\{1,\dots,k\}$ is an {\em injective $k$-coloring} if
any two vertices $u$ and $w$ with $f(u)=f(w)$ have no common neighbor. 
In other words, $f$ restricted to the neighborhood of any vertex is injective, which justifies the word ``injective'' in the name.
The minimum $k$ for which a graph $G$ admits an injective $k$-coloring is the {\em injective chromatic number} of $G$, and is denoted by $\chi_{i}(G)$. An injective $k$-coloring of a graph $G$, where $k=\chi_i(G)$, is a {\em $\chi_i$-coloring of $G$}. Injective colorings were introduced by Hahn, Kratochv\' il, \v{S}ir\'{a}\v{n} and Sotteau~\cite{hkss}, and were studied by a number of authors from various aspects; see~\cite{BSY, ICGenMycielskain, bu-2009, cky-2010, deng-2021, dl, jin-2013, ls-2015, mo, pp, SSY, sy} for a short selection of papers on injective colorings.

Similar to the injective (vertex) coloring, an edge version of the concept was proposed in~\cite{cardoso2019iec}; see also~\cite{csilla} for a related earlier study. 
 An {\em injective edge $k$-coloring} of a graph $G$ is a coloring $c : E(G) \rightarrow \{1,\dots,k\}$, such that if $e_1=xy, e_2=yz$ and $e_3=zu$ are consecutive edges in $G$ (where $u=x$ is also allowed), then $c(e_1) \neq c(e_3)$. The {\em injective chromatic index} of a graph $G$, $\chi_i'(G)$, is the minimum $k$ such that $G$ admits an injective edge $k$-coloring. An injective edge $k$-coloring of $G$ with $k=\chi_i'(G)$ is a {\em $\chi_i'$-coloring of $G$}.
  The injective edge coloring was studied from variety of angles~\cite{cardoso2019iec, IECBCK, FerdjallahIEC2021, FoucaudIEC2021}. Not surprisingly, the decision version of injective chromatic number (resp.\ index) is an NP-complete problem; see~\cite{hkss} and~\cite{cardoso2019iec}, respectively. 

The family of Sierpi\'{n}ski graphs was introduced by Klav\v zar and Milutinovi\'{c} in~\cite{klavzar-1997} with motivation arising from the Switching Tower of Hanoi game. Sierpi\'nski graphs attained a lot of interest among researchers, which led in 2017 to an extensive survey paper~\cite{hinz-2017}; see also some recent papers addressing new aspects of study of this graph family~\cite{anitha-2025,balakrishnan-2022, menon-2023}. A comprehensive monograph~\cite{hanoi} presents 
the theoretical significance and diverse practical applications of the 
Tower of Hanoi Problem and related topics. 
In 2011, Gravier, Kov\v se, and Parreau~\cite{gravier-2011} introduced generalized Sierpi\'nski graphs, which are defined as follows. Given an arbitrary graph $G$ and a positive integer $n$, the {\em generalized Sierpi\'nski graph} $S_G^n$ is the graph with the vertex set $V(G)^n$, where two vertices $(u_1,\ldots, u_n)$ and $(v_1,\ldots, v_n)$ are adjacent if there exists $d\in [n]$ such that 
\begin{enumerate}[noitemsep]
    \item $u_i = v_i$ for $i<d$, 
    \item $u_dv_d\in E(G)$, and 
    \item $u_i = v_d$ and $v_i = u_d$ for $i>d$.
\end{enumerate}
(The notation $[n]=\{1,\ldots,n\}$ and $[n]_0=\{0,1,\ldots,n\}$  will be used throughout the paper.) 
We also use the expressions that $G$ is the {\em base graph} of $S_G^n$, and $S_G^n$ is the generalized Sierpi\'nski graph {\em over} $G$ of {\em dimension} $n$. 
See Figure~\ref{fig:S_C_4} depicting generalized Sierpi\'{n}ski graphs over the cycle $C_4$ for $n=1$, $n=2$ and $n=3$; note that $V(C_4)=[4]$ in the figure and brackets and commas are omitted. Intuitively, the graph $S_G^n$ is obtained from $|V(G)|$ copies of the graph $S_G^{n-1}$ by connecting the vertex $(u,\ldots, u)$ in the copy of $S_G^{n-1}$ that corresponds to $v\in V(G)$ (and this vertex is denoted by $(v,u,\ldots, u)$ in $S_G^n$) to the vertex $(v,\ldots, v)$ in the copy of $S_G^{n-1}$  that corresponds to $u\in V(G)$ (and this vertex is denoted by $(u,v,\ldots, v)$ in $S_G^n$) whenever $uv\in E(G)$.  It is easy to see that the maximum degree of $S_G^n$ is $\Delta + 1$, where $\Delta$ is  the maximum degree of $G$. Note that the (standard) Sierpi\'{n}ski graphs are those in which the base graph $G$ is a complete graph, that is, they are the graphs $S_{K_p}^n$, where $p\ge 3$ and $n$ are positive integers. Vertices $(u_1,\ldots, u_n)$, with $u_i=u$ for all $i\in [n]$, where $u\in V(G)$, are {\em extreme vertices} of $S_G^n$. Note that the distance between any two distinct extreme vertices in $S_G^n$ is at least $2^n-1$. This can be seen by combining the formula for distances in $S_{K_p}^n$ (see~\cite[Lemma 4]{klavzar-1997}) with the fact that $S_G^n$ is a spanning subgraph of $S_{K_p}^n$, where $p=|V(G)|$; see also~\cite{err}.
For any graph-theoretic notions not defined here we refer to West~\cite{we}.

\begin{figure}[ht!]
\begin{center}
\begin{tikzpicture}[scale=0.7,style=thick,x=1cm,y=1cm]
\def\vr{3pt}

\begin{scope}[xshift=-12cm, yshift=0cm] 
\coordinate(x1) at (0,0);
\coordinate(x2) at (0,1);
\coordinate(x3) at (1,1);
\coordinate(x4) at (1,0);
\draw (x1) -- (x2) -- (x3) -- (x4) -- (x1);

\foreach \i in {1,...,4}
{ 
\draw(x\i)[fill=white] circle(\vr);
}

\draw[below] (x1)++(0,-0.1) node {$1$};
\draw[above] (x2)++(0,0) node {$2$};
\draw[above] (x3)++(0,0) node {$3$};
\draw[below] (x4)++(0,-0.1) node {$4$};
\end{scope}

\begin{scope}[xshift=-7cm, yshift=0cm] 
\coordinate(x1) at (0,0);
\coordinate(x2) at (1,0);
\coordinate(x3) at (2,0);
\coordinate(x4) at (3,0);
\coordinate(x5) at (0,1);
\coordinate(x6) at (1,1);
\coordinate(x7) at (2,1);
\coordinate(x8) at (3,1);
\coordinate(x9) at (0,2);
\coordinate(x10) at (1,2);
\coordinate(x11) at (2,2);
\coordinate(x12) at (3,2);
\coordinate(x13) at (0,3);
\coordinate(x14) at (1,3);
\coordinate(x15) at (2,3);
\coordinate(x16) at (3,3);
\draw (x1) -- (x2) -- (x6) -- (x5) -- (x1);
\draw (x2) -- (x3) -- (x4) -- (x8) -- (x7) -- (x3);
\draw (x8) -- (x12) -- (x16) -- (x15) -- (x11) -- (x12);
\draw (x15) -- (x14) -- (x13) -- (x9) -- (x10) -- (x14);
\draw (x5) -- (x9);
\foreach \i in {1,...,16}
{ 
\draw(x\i)[fill=white] circle(\vr);
}

\draw[below] (x1)++(0,-0.1) node {\small $11$};
\draw[below] (x2)++(0,-0.1) node {\small $14$};
\draw[below] (x3)++(0,-0.1) node {\small $41$};
\draw[below] (x4)++(0,-0.1) node {\small $44$};
\draw[above] (x13)++(0,0.1) node {\small $22$};
\draw[above] (x14)++(0,0.1) node {\small $23$};
\draw[above] (x15)++(0,0.1) node {\small $32$}; 
\draw[above] (x16)++(0,0.1) node {\small $33$};
\draw[left] (x5)++(0,0) node {\small $12$};
\draw[left] (x9)++(0,0) node {\small $21$};
\draw[right] (x8)++(0,0) node {\small $43$};
\draw[right] (x12)++(0,0) node {\small $34$};
\draw[right] (x10)++(-0.05,0) node {\small $24$};
\draw[below] (x11)++(0,0.) node {\small $31$};
\draw[left] (x7)++(0.,0) node {\small $42$};
\draw[above] (x6)++(0,0) node {\small $13$};
\end{scope}

\begin{scope}[xshift=0cm, yshift=0cm] 
\coordinate(x1) at (0,0);
\coordinate(x2) at (1,0);
\coordinate(x3) at (2,0);
\coordinate(x4) at (3,0);
\coordinate(x5) at (0,1);
\coordinate(x6) at (1,1);
\coordinate(x7) at (2,1);
\coordinate(x8) at (3,1);
\coordinate(x9) at (0,2);
\coordinate(x10) at (1,2);
\coordinate(x11) at (2,2);
\coordinate(x12) at (3,2);
\coordinate(x13) at (0,3);
\coordinate(x14) at (1,3);
\coordinate(x15) at (2,3);
\coordinate(x16) at (3,3);
\draw (x1) -- (x2) -- (x6) -- (x5) -- (x1);
\draw (x2) -- (x3) -- (x4) -- (x8) -- (x7) -- (x3);
\draw (x8) -- (x12) -- (x16) -- (x15) -- (x11) -- (x12);
\draw (x15) -- (x14) -- (x13) -- (x9) -- (x10) -- (x14);
\draw (x5) -- (x9);
\draw (3,0) -- (4,0);
\draw (0,3) -- (0,4);
\draw (7,3) -- (7,4);
\draw (3,7) -- (4,7);
\foreach \i in {1,...,16}
{ 
\draw(x\i)[fill=white] circle(\vr);
}
\draw[below] (x1)++(0,0.06) node {\tiny $111$};
\draw[below] (x2)++(0,0.06) node {\tiny $114$};
\draw[below] (x3)++(0,0.06) node {\tiny $141$};
\draw[below] (x4)++(0,0.06) node {\tiny $144$};
\draw[left] (x5)++(0,0.0) node {\tiny $112$};
\draw[above] (x6)++(0,0) node {\tiny $113$};
\draw[above] (x7)++(0,-0.001) node {\tiny $142$};
\draw[above] (x8)++(-.08,0.0) node {\tiny $143$};
\draw[left] (x9)++(0,0.06) node {\tiny $121$};
\draw[below] (x10)++(0,0.06) node {\tiny $124$};
\draw[below] (x11)++(0,0.06) node {\tiny $131$};
\draw[below] (x12)++(-0.08,0.06) node {\tiny $134$};
\draw[left] (x13)++(0,0) node {\tiny $122$};
\draw[above] (x14)++(0,0) node {\tiny $123$};
\draw[above] (x15)++(0,0) node {\tiny $132$};
\draw[above] (x16)++(0,0.01) node {\tiny $133$};

\end{scope}

\begin{scope}[xshift=4cm, yshift=0cm] 
\coordinate(x1) at (0,0);
\coordinate(x2) at (1,0);
\coordinate(x3) at (2,0);
\coordinate(x4) at (3,0);
\coordinate(x5) at (0,1);
\coordinate(x6) at (1,1);
\coordinate(x7) at (2,1);
\coordinate(x8) at (3,1);
\coordinate(x9) at (0,2);
\coordinate(x10) at (1,2);
\coordinate(x11) at (2,2);
\coordinate(x12) at (3,2);
\coordinate(x13) at (0,3);
\coordinate(x14) at (1,3);
\coordinate(x15) at (2,3);
\coordinate(x16) at (3,3);
\draw (x1) -- (x2) -- (x6) -- (x5) -- (x1);
\draw (x2) -- (x3) -- (x4) -- (x8) -- (x7) -- (x3);
\draw (x8) -- (x12) -- (x16) -- (x15) -- (x11) -- (x12);
\draw (x15) -- (x14) -- (x13) -- (x9) -- (x10) -- (x14);
\draw (x5) -- (x9);
\foreach \i in {1,...,16}
{ 
\draw(x\i)[fill=white] circle(\vr);
}
\draw[below] (x1)++(0,0.06) node {\tiny $411$};
\draw[below] (x2)++(0,0.06) node {\tiny $414$};
\draw[below] (x3)++(0,0.06) node {\tiny $441$};
\draw[below] (x4)++(0,0.06) node {\tiny $444$};
\draw[above] (x5)++(-0.08,0.0) node {\tiny $412$};
\draw[above] (x6)++(0,0) node {\tiny $413$};
\draw[above] (x7)++(0,-0.001) node {\tiny $442$};
\draw[right] (x8)++(0,0.05) node {\tiny $443$};
\draw[below] (x9)++(-0.08,0.06) node {\tiny $421$};
\draw[below] (x10)++(0,0.06) node {\tiny $424$};
\draw[below] (x11)++(0,0.06) node {\tiny $431$};
\draw[right] (x12)++(0,0.06) node {\tiny $434$};
\draw[above] (x13)++(0,0) node {\tiny $422$};
\draw[above] (x14)++(0,0) node {\tiny $423$};
\draw[above] (x15)++(0,0) node {\tiny $432$};
\draw[right] (x16)++(0,0.01) node {\tiny $433$};

\end{scope}

\begin{scope}[xshift=0cm, yshift=4cm] 
\coordinate(x1) at (0,0);
\coordinate(x2) at (1,0);
\coordinate(x3) at (2,0);
\coordinate(x4) at (3,0);
\coordinate(x5) at (0,1);
\coordinate(x6) at (1,1);
\coordinate(x7) at (2,1);
\coordinate(x8) at (3,1);
\coordinate(x9) at (0,2);
\coordinate(x10) at (1,2);
\coordinate(x11) at (2,2);
\coordinate(x12) at (3,2);
\coordinate(x13) at (0,3);
\coordinate(x14) at (1,3);
\coordinate(x15) at (2,3);
\coordinate(x16) at (3,3);
\draw (x1) -- (x2) -- (x6) -- (x5) -- (x1);
\draw (x2) -- (x3) -- (x4) -- (x8) -- (x7) -- (x3);
\draw (x8) -- (x12) -- (x16) -- (x15) -- (x11) -- (x12);
\draw (x15) -- (x14) -- (x13) -- (x9) -- (x10) -- (x14);
\draw (x5) -- (x9);
\foreach \i in {1,...,16}
{ 
\draw(x\i)[fill=white] circle(\vr);
}
\draw[left] (x1)++(0,0.0005) node {\tiny $211$};
\draw[below] (x2)++(0,0.06) node {\tiny $214$};
\draw[below] (x3)++(0,0.06) node {\tiny $241$};
\draw[below] (x4)++(0,0.06) node {\tiny $244$};
\draw[left] (x5)++(0,0.0005) node {\tiny $212$};
\draw[above] (x6)++(0,0.0005) node {\tiny $213$};
\draw[above] (x7)++(0,0.0005) node {\tiny $242$};
\draw[above] (x8)++(-0.08,0.0005) node {\tiny $243$};
\draw[left] (x9)++(0,0.06) node {\tiny $221$};
\draw[below] (x10)++(0,0.06) node {\tiny $224$};
\draw[below] (x11)++(0,0.06) node {\tiny $231$};
\draw[below] (x12)++(-0.08,0.06) node {\tiny $234$};
\draw[above] (x13)++(0,0) node {\tiny $222$};
\draw[above] (x14)++(0,0) node {\tiny $223$};
\draw[above] (x15)++(0,0) node {\tiny $232$};
\draw[above] (x16)++(0,0) node {\tiny $233$};

\end{scope}

\begin{scope}[xshift=4cm, yshift=4cm] 
\coordinate(x1) at (0,0);
\coordinate(x2) at (1,0);
\coordinate(x3) at (2,0);
\coordinate(x4) at (3,0);
\coordinate(x5) at (0,1);
\coordinate(x6) at (1,1);
\coordinate(x7) at (2,1);
\coordinate(x8) at (3,1);
\coordinate(x9) at (0,2);
\coordinate(x10) at (1,2);
\coordinate(x11) at (2,2);
\coordinate(x12) at (3,2);
\coordinate(x13) at (0,3);
\coordinate(x14) at (1,3);
\coordinate(x15) at (2,3);
\coordinate(x16) at (3,3);
\draw (x1) -- (x2) -- (x6) -- (x5) -- (x1);
\draw (x2) -- (x3) -- (x4) -- (x8) -- (x7) -- (x3);
\draw (x8) -- (x12) -- (x16) -- (x15) -- (x11) -- (x12);
\draw (x15) -- (x14) -- (x13) -- (x9) -- (x10) -- (x14);
\draw (x5) -- (x9);
\foreach \i in {1,...,16}
{ 
\draw(x\i)[fill=white] circle(\vr);
}

\draw[below] (x1)++(0,0.0005) node {\tiny $311$};
\draw[below] (x2)++(0,0.06) node {\tiny $314$};
\draw[below] (x3)++(0,0.06) node {\tiny $341$};
\draw[right] (x4)++(0,0.06) node {\tiny $344$};
\draw[above] (x5)++(-0.08,0.0005) node {\tiny $312$};
\draw[above] (x6)++(0,0.0005) node {\tiny $313$};
\draw[above] (x7)++(0,0.0005) node {\tiny $342$};
\draw[right] (x8)++(0,0.0005) node {\tiny $343$};
\draw[below] (x9)++(-0.08,0.06) node {\tiny $321$};
\draw[below] (x10)++(0,0.06) node {\tiny $324$};
\draw[below] (x11)++(0,0.06) node {\tiny $331$};
\draw[right] (x12)++(0,0.06) node {\tiny $334$};
\draw[above] (x13)++(0,0) node {\tiny $322$};
\draw[above] (x14)++(0,0) node {\tiny $323$};
\draw[above] (x15)++(0,0) node {\tiny $332$};
\draw[above] (x16)++(0,0) node {\tiny $333$};

\end{scope}

\end{tikzpicture}
\caption{$S_{C_4}^1$ (left), $S_{C_4}^2$ (middle) and  $S_{C_4}^3$ (right).}
\label{fig:S_C_4}
\end{center}
\end{figure}
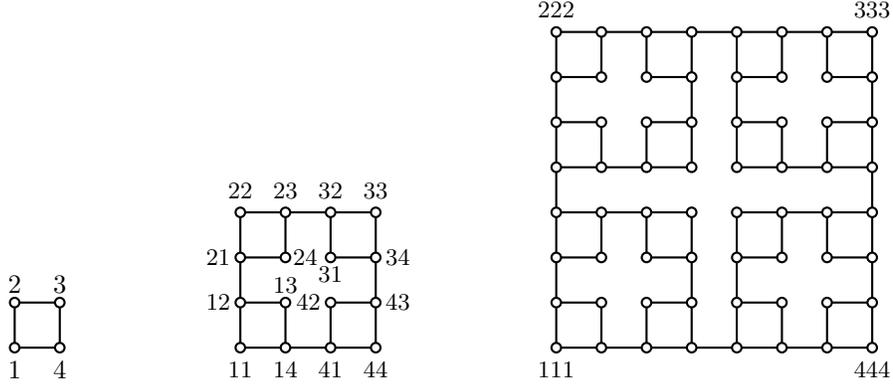

Many graph invariants have been investigated in generalized Sierpi\'{n}ski graphs; see a selection of references~\cite{bf, er, erv, gs, kz, korze-2019, pjk, rre}. In a recent paper~\cite{bksy}, the authors considered injective colorings of (generalized) Sierpi\'{n}ski graphs. Their main focus was on the standard Sierpi\' nski graphs, that is, when the base graph is a complete graph.  They proved that $\chi_i(S_{K_p}^n)=p$ holds for any $p\ge 3$ and any $n$. In addition, for $n\ge 2$ they established that $\chi_i(S_{C_4}^n) = 3$, and suggested a thorough investigation of the injective chromatic number of generalized Sierpi\'nski graphs. In particular, concerning the generalized Sierpi\'{n}ski graphs over cycles, they suspected that $\chi_i(S_{C_k}^n) = 3$ holds for any $n\ge 2$ and $k\ge 5$.  

In this paper, we follow the suggestion from~\cite{bksy} and investigate the injective chromatic numbers in  Sierpi\'{n}ski graphs. We confirm the suspicion of the authors in~\cite{bksy} about the value of the injective chromatic number of generalized Sierpi\'{n}ski graphs over cycles. The equality $\chi_i(S_{C_k}^n) = 3$, where $n\ge 2$ and $k\ge 3$, is obtained by using a general result that we prove on the injective chromatic number of the generalized Sierpi\'{n}ski graph $S_G^n$, which limits its value to only two possibilities  depending on $\chi_i(S_G^2)$. Then, we improve the mentioned general result by showing that $$\chi_i(G)\le\chi_i(S_G^n)\le \chi_i(G)+1$$
holds for any graph $G$ and any positive integer $n$, which remarkably restricts the options for the value of $\chi_i(S_G^n)$ to only two possibilities. The described study of injective (vertex) colorings of generalized Sierpi\'{n}ski graphs is given in Section~\ref{sec:inj}.

Section~\ref{sec:injedge} is devoted to injective edge colorings of generalized Sierpi\'{n}ski graphs. It turns out that the edge version of injective coloring is more difficult at least in generalized Sierpi\'{n}ski graphs. We determine the exact values of the injective chromatic indices for the standard Sierpi\'{n}ski graphs $S_3^n$ (that is, the generalized Sierpi\'{n}ski graphs whose base graph is $K_3$). In the proof that $\chi_i'(S_{K_3}^n)=5$ when $n\ge 3$ we use two types of injective edge colorings, which need to be combined in an appropriate way so that both types of colorings are obtained for all $n$. It seems that (many) triangles in a graph make difficulties in determining the injective chromatic index, hence in the last main result we consider triangle-free graphs. If $G$ is triangle-free then we obtain a similar bound for $S_G^n$ as in the case of injective (vertex) colorings, only this time the bound is based on the value of the $3$-dimensional Sierpi\'{n}ski graph over $G$. Notable, we prove that $$\chi_i'(S_G^3)\le \chi_i'(S_G^n)\le \chi_i'(S_G^3)+1$$ for any $n\ge 3$ and any triangle-free graph $G$, and add sufficient conditions for a $\chi_i'$-coloring of $S_G^3$ which enforce that $\chi_i'(S_G^n)=\chi_i'(S_G^3)$ holds for all $n$. The latter conditions enable us to prove the formulas $\chi_i'(S_{C_4}^n)=3$, $\chi_i'(S_{C_5}^n)=4$ and $\chi_i'(S_{C_6}^n)=3$, which hold for all $n\ge 2$.  

\section{Injective coloring}
\label{sec:inj}

We start with an upper bound on the injective chromatic number of $S_G^n$ using  
$\chi_i(S_G^2)$. 

\begin{theorem}
For any graph $G$ and any positive integer $n\ge 2$,
	$$\chi_i(S_G^2)\le \chi_i(S_G^n)\le \chi_i(S_G^2)+1.$$
\vskip -0.2cm
\label{thm:SG2}
\end{theorem}
\begin{proof}
Clearly, $\chi_i(S_G^n)\geq \chi_i(S_G^2)$, since $S_G^2$ is a subgraph of $S_G^n$. Let $f:V(S_G^2)\rightarrow [k]$ be a $\chi_i$-coloring of $S_G^2$. Then, for $n\ge 3$, let $f_1:V(S_G^n)\rightarrow [k+1]$ be defined as follows:
	\begin{itemize}[noitemsep]
		\item $f_1(u)=f(u_{n-1},u_{n}) \text{ if } u_{n-1}\neq u_n$,
		\item $f_1(u)=k+1 \text{ if }  u_{n-1}= u_n$,
	\end{itemize}
where the notation $u=(u_1,\ldots,u_n)$ is used.

Considering the natural partition of $S_G^n$ into $|V(G)|^{n-2}$ copies of $S_G^2$ note that all neighbors of a vertex $w\in V(S_G^n)$ with $w_{n-1}\ne w_{n}$ are in the copy of $S_G^2$ to which $w$ belongs. Hence, the colors given by $f_1$ to vertices in $N_{S_G^n}(w)$ are pairwise distinct, since $f$ is an injective coloring and there is at most one vertex $u$ in $N_{S_G^n}(w)$ with $u_{n-1}=u_n$ (which is given color $k+1$ by $f_1$). On the other hand, if $w_{n-1}= w_n$ for a vertex $w\in V(S_G^n)$, then it has exactly one neighbor in another copy of $S_G^2$, while all other vertices are in the same copy of $S_G^2$ as $w$. We again infer that the colors given by $f_1$ to vertices in $N_{S_G^n}(w)$ are pairwise distinct. Thus, $f_1$ is an injective coloring of $S_G^n$ with $\chi_i(S_G^2)+1$ colors, therefore $\chi_i(S_G^n)\le \chi_i(S_G^2)+1$.
\end{proof}

The following theorem provides a useful sufficient condition for the equality $\chi_i(S_G^n)= \chi_i(S_G^2)$. 

\begin{theorem}
\label{thm:SG2suff}
If there exists a $\chi_i$-coloring $f$ of $S_G^2$ such that all extreme vertices of $S_G^2$ receive the same color and for each extreme vertex $(u,u)$, $u\in V(G)$, the color $f(u,u)$ does not appear in any of its neighbors,  then $\chi_i(S_G^n)= \chi_i(S_G^2)$ for all $n\geq 2$.
\end{theorem}
\begin{proof}
The proof is based on a similar idea as the proof of Theorem~\ref{thm:SG2}, also considering the partition of $S_G^n$ into $|V(G)|^{n-2}$ copies of $S_G^2$. Notably, assuming that $f:V(S_G^2)\rightarrow [\chi_i(S_G^2)]$ is an injective coloring with the property from the statement of the theorem, we define the coloring $f_1:V(S_G^n)\rightarrow [\chi_i(S_G^2)]$ by letting $$f_1(w)=f(w_{n-1},w_n).$$ 
For vertices $w\in V(S_G^n)$ with $w_{n-1}\ne w_{n}$ we derive that vertices in $N_{S_G^n}(w)$ are given pairwise distinct colors by $f_1$, since their colors coincide with the corresponding colors in $S_G^2$ given by $f$. Now, assume that $w\in V(S_G^n)$ has $w_{n-1}=w_{n}$. By definition of $f_1$ and the assumption concerning coloring $f$, the color $f(w,w)$ does not appear in any of the neighbors of $w$ in the copy of $S_G^2$ in which $w$ lies, yet the neighbor of $w$ in another copy of $S_G^2$ in $S_G^n$ is given the color $f(w,w)$ by $f_1$. Therefore, $f_1$ is an injective coloring of $S_G^n$ and the stated result readily follows. 
\end{proof}	

By using Theorem~\ref{thm:SG2suff}, we now prove the announced result about Sierpi\'nski graphs whose base graph is a cycle. 

\begin{proposition}
For any $n\ge 2$ and any $k\ge 3$ we have $\chi_i(S_{C_k}^n) = 3$.
\end{proposition}
\begin{proof}
The cases when $k\in\{3,4\}$ were established in~\cite{bksy}, thus assume $k\ge 5$, and let $V(C_k)=[k-1]_0=\{0,1,\ldots, k-1\}$. Since $\Delta(S_{C_k}^n)=3$ for all $n\ge 2$ and $k\ge 3$, we immediately infer that $\chi_i(S_{C_k}^n)\ge 3$. Thus, to prove that $\chi_i(S_{C_k}^n) = 3$ it suffices to find an injective coloring of $S_{C_k}^2$ with three colors that satisfies the conditions in Theorem~\ref{thm:SG2suff}.

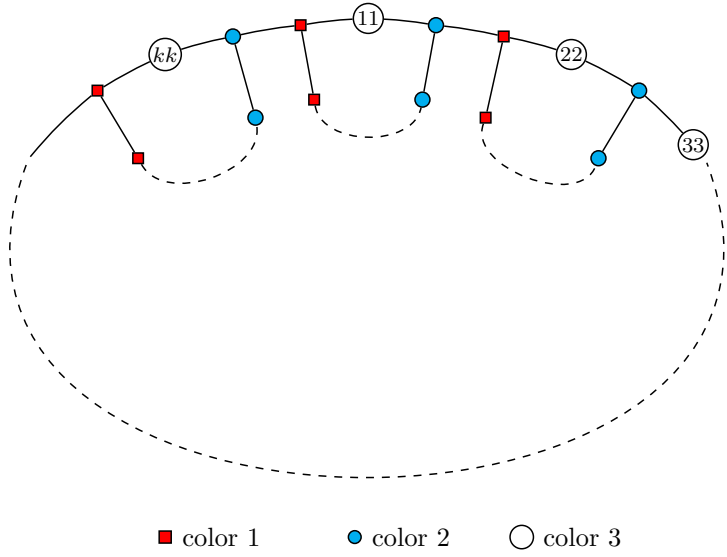
\begin{figure}[H]
    \centering
    \begin{tikzpicture}[scale=0.6,inner sep=2pt, line width=0.2mm]
            \draw (0,7.5) node(1) [circle,draw,fill=white,inner sep=0.5pt] {{\fontsize{8pt}{0pt}\selectfont{$11$}}}
                  (-1.5,7.35) node(2) [rectangle,draw,fill=red] {}
                  (1.5,7.35) node(3) [circle,draw,fill=cyan] {}
                  (-1.2,5.7) node(12) [rectangle,draw,fill=red] {}
                  (1.2,5.7) node(13) [circle,draw,fill=cyan] {};    
            \draw [-,bend right=3] (1) to (2);
            \draw [-,bend right=3] (3) to (1);

            \draw (-3,7.1) node(4) [circle,draw,fill=cyan] {}
                  (3,7.1) node(5) [rectangle,draw,fill=red] {}
                  (-2.5,5.3) node(11) [circle,draw,fill=cyan] {}
                  (-5.1,4.4) node(10) [rectangle,draw,fill=red] {}
                  (-4.5,6.7) node(6) [circle,draw,fill=white,inner sep=0.5pt] {{\fontsize{8pt}{0pt}\selectfont{$kk$}}}
                  (4.5,6.7) node(7) [circle,draw,fill=white,inner sep=0.5pt] {{\fontsize{8pt}{0pt}\selectfont{$22$}}}
                  (-6,5.9) node(8) [rectangle,draw,fill=red] {}
                  (6,5.9) node(9) [circle,draw,fill=cyan] {}
                  (2.6,5.3) node(14) [rectangle,draw,fill=red] {}
                  (5.1,4.4) node(15) [circle,draw,fill=cyan] {}
                  (-7.6,4.3) node(16){}
                  (7.2,4.7) node(17) [circle,draw,fill=white,inner sep=0.5pt] {{\fontsize{8pt}{0pt}\selectfont{$33$}}};
             \draw [-,bend right=80,dashed] (10) to (11); 
             \draw [-,bend right=80,dashed] (12) to (13);
             \draw [-,bend right=90,dashed] (14) to (15);
             \draw [-,bend left=3] (8) to (6);
             \draw [-,bend left=3] (3) to (5);
             \draw [-,bend left=3] (6) to (4);
             \draw [-,bend left=3] (5) to (7);
             \draw [-,bend right=3] (9) to (7);
             \draw [-,bend left=3] (4) to (2);
             \draw [-,bend right=5] (8) to (16);
             \draw [-,bend left=3] (9) to (17);
             \draw[-,dashed] (-7.58,4.25) .. controls (-11,-5) and (11,-5) .. (7.5,4.3);
             \draw[] (8) to (10)
                     (4) to (11)
                     (2) to (12)
                     (3) to (13)
                     (5) to (14)
                     (9) to (15);
        \draw (-4.5,-4) node() [inner sep=0.8mm,draw,rectangle, fill=red]{};
	\draw (-3.4,-4) node[] {\, color 1};
	\draw (-0.3,-4) node() [inner sep=0.6mm,draw,circle,fill=cyan]{};
	\draw (0.8,-4) node[] {\, color 2};
	\draw (3.4,-4) node() [inner sep=1.1mm,draw, circle, fill=white]{};
	\draw (4.6,-4) node[] {\, color 3};                      
\end{tikzpicture}
    \caption{A sketch of the injective coloring of $S_{C_k}^2$.}
    \label{General graph of S{C_k}2}
\end{figure}
The desired coloring $f:V(S_{C_k}^2)\to [3]$ is obtained by letting $f(i,i)=3$ for all $i\in [k-1]_0$, $f(i,i+1)=f(i,i+2)=2$ and $f(i,i-1)=f(i,i-2)=1$ for all $i\in [k-1]_0$ where all values are taken with respect to modulo $k$. (See Figure~\ref{General graph of S{C_k}2}.) If $k=5$, the coloring is already determined. Otherwise, for $k\ge 6$, we alternate the color pattern $2211$ along the cycle and, if necessary, complete it with either one or two vertices given color $3$.  More precisely, the following colors are applied between vertices $(i,i-3)$ and $(i,i+3)$:

\begin{itemize}[noitemsep]     
    \item $k\equiv 0 \mod 4$:
          \begin{itemize}[noitemsep]
              \item Use the repeating pattern $22112211 \cdots$ 
              \item Assign colors $3, 3, 2$ to the final three vertices in the sequence.
          \end{itemize}
    \item $k\equiv 1 \mod 4$:
          \begin{itemize}
              \item Use the repeating pattern $22112211 \cdots$ until all vertices are colored.
          \end{itemize}          
    \item $k\equiv 2 \mod 4$:
          \begin{itemize}[noitemsep]
              \item Use the repeating pattern $22112211 \cdots$ 
              \item Assign color $3$ to the final vertex.
          \end{itemize}  
    \item $k\equiv 3 \mod 4$:
          \begin{itemize}[noitemsep]
              \item Use the repeating pattern $22112211 \cdots$ 
              \item Assign colors $3, 3$ to the final two vertices.
          \end{itemize}  
\end{itemize}
Note that the above patterns provide an injective coloring of $S_{C_k}^2$ for all values of $k\ge 6$ by using three colors. The coloring also satisfies the conditions from Theorem~\ref{thm:SG2suff}, notably, it is a $\chi_i$-coloring of $S_{C_k}^2$ in which all extreme vertices receive the same color (namely, color $3$) and for each extreme vertex $(u,u)$, $u\in V(G)$, the color $f(u,u)$ does not appear in any of its neighbors. Thus, $\chi_i(S_{C_k}^n)=3$, as claimed.
\end{proof}

Using Theorem~\ref{thm:SG2} we infer our main result, which bounds $\chi_i(S_G^n)$ with respect to $\chi_i(G)$.

\begin{theorem}
\label{thm:main}
For any graph $G$ and positive integer $n$, $$\chi_i(G)\le\chi_i(S_G^n)\le \chi_i(G)+1,$$ and the bounds are sharp in the sense that there are  infinite families of graphs $G$ such that $\chi_i(S_G^n)=\chi_i(G)$ for all $n\in \mathbb{N}$, and $\chi_i(S_G^n)=\chi_i(G)+1$ for all $n\ge 2$, respectively. 
\end{theorem}
\begin{proof}
The inequality $\chi_i(G)\le\chi_i(S_G^n)$ is trivial. 
For the other inequality, let us first prove that $\chi_i(S_G^2)\le \chi_i(G)+1$. To see this, let $f:V(G)\rightarrow [k]$ be a $\chi_i$-coloring of $G$, and define a coloring $f_1:V(S_G^2)\rightarrow [k+1]$ as follows:

\begin{align*}
f_1(u_1,u_2)=
\begin{cases}
f(u_2) & \mbox{if } (u_1,u_2) \mbox{ is not an extreme vertex of } S_G^2,\\
k+1 & \mbox{if}\ u_1=u_2.
\end{cases}
\end{align*}
To verify that $f_1$ is an injective coloring of $S_G^2$ consider the natural partition of the graph into $|V(G)|$ induced subgraphs isomorphic to $G$. Note that a vertex $(u_1,u_2)$ has a neighbor in another subgraph with respect to this partition if and only if $u_1u_2\in E(G)$. Thus, if $u_1u_2\notin E(G)$, then the colors of vertices assigned by $f_1$ in $N_{S_G^2}(u_1,u_2)$ are pairwise distinct, since they are the same as the colors assigned by $f$ to the vertices in $N_G(u_2)$. On the other hand, if $u_1u_2\in E(G)$, then $(u_1,u_2)$ is adjacent also to $(u_2,u_1)$, which belongs to another subgraph of the mentioned partition. As in the previous case we note that no two neighbors of $(u_1,u_2)$ that lie in the same copy of $G$ in $S_G^2$ can receive the same color by $f_1$. Suppose that $f_1(u_2,u_1)=f_1(u_1,u_3)$, where $(u_1,u_3)$ is a neighbor of $(u_1,u_2)$. Since $f_1(u_2,u_1)\ne k+1$, this is only possible if $u_3\ne u_1$. However, $(u_1,u_3)(u_1,u_2)\in E(S_G^2)$ implies $u_2u_3\in E(G)$, which means that $u_1$ and $u_3$ have $u_2$ as a common neighbor in $G$. On the other hand, the former equality means $f(u_1)=f(u_3)$, which is a contradiction to  $f$ being an injective coloring. Thus, $f_1$ is indeed an injective coloring of $S_G^2$, and so $\chi_i(S_G^2)\le \chi_i(G)+1$. 

To see that $\chi_i(S_G^n)\le \chi_i(G)+1$ for any $n\ge 3$, consider again the coloring $f_1:V(S_G^2)\rightarrow [k+1]$ as defined in the previous paragraph.
Note that $f_1(u)=k+1$ if and only if $u$ is an extreme vertex of $S_G^2$. Since extreme vertices are distance at least $3$ in $S_G^2$, we infer that $f_1$ is an injective coloring of $S_G^2$ such that all extreme vertices receive the same color and for each vertex $(u,u)$, $u\in V(G)$, the color $f_1(u,u)$ does not appear in any of its neighbors. If $\chi_i(S_G^2)=k+1$, then by using Theorem~\ref{thm:SG2suff} we immediately infer that $\chi_i(S_G^n)= \chi_i(S_G^2)$ for all $n\geq 2$. Thus, $\chi_i(S_G^n)=k+1=\chi_i(G)+1$. On the other hand, if $\chi_i(S_G^2)<k+1$, we infer $\chi_i(S_G^2)=\chi_i(G)=k$. Now, Theorem~\ref{thm:SG2} gives $\chi_i(S_G^n)\le \chi_i(S_G^2)+1=k+1$ for all $n\ge 2$, and the proof is complete. 

For the sharpness of the bound, first note that there exist graphs $G$ with $\chi_i(S_G^n) = \chi_i(G)$ for all positive integers $n$. Consider complete graphs $K_p$ and recall from~\cite{bksy} that $\chi_i(S_{K_p}^n) = p = \chi_i(K_p)$ for all $p\ge 3$.  Similarly, for a cycle $C_k$ with $k\not\equiv 0 \pmod 4$ we have  $\chi_i(S_{C_k}^n) = 3 = \chi_i(C_k)$.
On the other hand, recall from~\cite{hkss} that $\chi_i(C_k)=2$ holds if $k\equiv 0 \pmod 4$. Therefore, $\chi_i(S_{C_k}^n) = 3=\chi_i(G)+1$ in this case.
\end{proof}

While Theorem~\ref{thm:main} is conclusive in the sense that there are only two possible values of $\chi_i(S_G^n)$ for all positive integers $n$, namely $\chi_i(G)$ and $\chi_i(G)+1$, it would also be interesting to characterize the graphs $G$ in both of these families. We pose this as a problem.

\begin{problem}
For which graphs $G$ is $\chi_i(S_G^n)=\chi_i(G)$ for all positive integers $n$?
\end{problem}

The problem is likely hard, but even some partial answers restricted to some well known classes of graphs might be interesting. 

\section{Injective edge coloring}
\label{sec:injedge}

Let us first introduce the notion of a common edge that will be used in what follows. Considering three consecutive edges $e_1=xy,e_2=yz$ and $e_3=zu$, where $u=x$ is not excluded, the edge $e_2$ is a {\em common edge} of $e_1$ and $e_3$. 
Thus, a coloring $c:E(G)\to [k]$ is an injective edge coloring of $G$ if $c(e)=c(f)$ implies that edges $e$ and $f$ have no common edge. 



We start with the injective edge colorings of the generalized Sierpi\'nski graphs over the complete graph $K_3$, that is, the standard Sierpi\'{n}ski graphs $S_3^n$. It turns out that determining the injective chromatic index of $S_3^n$ is more challenging than it was for its vertex counterpart. 
\begin{theorem}
$\chi_i'(S_3^n)= \begin{cases}
	           3, ~~\textrm{if}~ n=1,\\
	           4, ~~\textrm{if}~ n=2,\\
	           5, ~ ~\textrm{if}~ n\geq 3. 
	          \end{cases}$
\end{theorem}
\begin{proof}
Clearly, $S_3^1 \cong K_3$, and $\chi_i'(S_3^1)=3$. For $n=2$, consider the graph $H$ depicted in Figure~\ref{ThesubgraphsofS3n}(a). The graph $H$ is a subgraph of $S_3^n$ for $n\geq 2$, which needs at least 4 colors in an injective edge coloring. Indeed, it is necessary to color the colored edges with 4 different colors, because of the adjacencies
with the dashed edges. Thus $\chi_i'(S_3^2)\geq 4$. On the other hand, Figure~\ref{IECSierpenskiofC3}(b) provides an injective edge coloring of $S_3^2$ with 4 colors showing that $\chi_i'(S_3^2)= 4$.

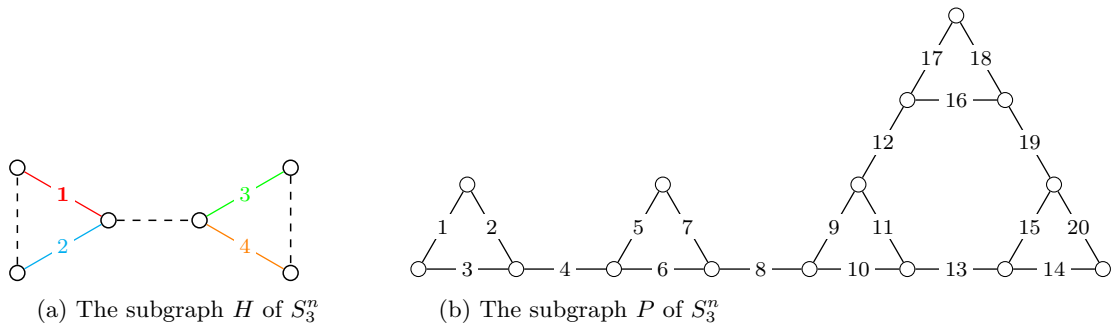
\begin{figure}[h!]
            \begin{subfigure}[b]{0.29\linewidth}
            \begin{tikzpicture}[scale=0.7,inner sep=2pt, line width=0.2mm]
            \draw (0,0) node(1) [circle,draw] {}
                  (0,2) node(2) [circle,draw] {}
                  (1.73,1) node(3) [circle,draw] {}
                  (3.46,1) node(4) [circle,draw] {}
                  (5.19,2) node(5) [circle,draw] {}
                  (5.19,0) node(6) [circle,draw] {};

                  \Edge[color=red, lw=0.5pt,label={{\fontsize{8pt}{8pt}\selectfont \textcolor{red}{\bf{1}}}}](2)(3)
                  \Edge[color=cyan, lw=0.5pt,label={{\fontsize{8pt}{8pt}\selectfont \textcolor{cyan}{2}}}](3)(1)
                  \Edge[color=green, lw=0.5pt,label={{\fontsize{8pt}{8pt}\selectfont \textcolor{green}{3}}}](4)(5)
                  \Edge[color=orange, lw=0.5pt,label={{\fontsize{8pt}{8pt}\selectfont \textcolor{orange}{4}}}](4)(6)

                  \draw[dashed] (1) to (2);
                  \draw[dashed] (3) to (4)
                          (5) to (6);
            \end{tikzpicture}    
            \caption{The subgraph $H$ of $S_3^n$}
            \end{subfigure}
            \hspace{.6cm}
\begin{subfigure}[b]{0.32\linewidth}
\begin{tikzpicture}[scale=0.58,inner sep=2pt, line width=0.2mm]    
            \draw (0,0) node(1) [circle,draw] {}
                  (2,0) node(2) [circle,draw] {}
                  (4,0) node(3) [circle,draw] {}
                  (6,0) node(4) [circle,draw] {}
                  (1,1.73) node(5) [circle,draw] {}
                  (5,1.73) node(6) [circle,draw] {}
                  (2,3.46) node(7) [circle,draw] {}
                  (4,3.46) node(8) [circle,draw] {}
                  (3,5.19) node(9) [circle,draw] {};

            \Edge[color=red,lw=0.5pt,label={{\fontsize{8pt}{8pt}\selectfont \textcolor{red}{{\bf 1}}}}](1)(5)
            \Edge[color=red,lw=2.5pt,label={{\fontsize{8pt}{8pt}\selectfont \textcolor{red}{{\bf 1}}}}](3)(6)
            \Edge[color=red,lw=2.5pt,label={{\fontsize{8pt}{8pt}\selectfont \textcolor{red}{{\bf 1}}}}](6)(8)

            \Edge[color=cyan,lw=0.5pt,label={{\fontsize{8pt}{8pt}\selectfont \textcolor{cyan}{2}}}](1)(2)
            \Edge[color=cyan,lw=2.5pt,label={{\fontsize{8pt}{8pt}\selectfont \textcolor{cyan}{2}}}](4)(6)
            \Edge[color=cyan,lw=0.5pt,label={{\fontsize{8pt}{8pt}\selectfont \textcolor{cyan}{2}}}](7)(9)

            \Edge[color=green,lw=0.5pt,label={{\fontsize{8pt}{8pt}\selectfont \textcolor{green}{3}}}](3)(4)
            \Edge[color=green,lw=0.5pt,label={{\fontsize{8pt}{8pt}\selectfont \textcolor{green}{3}}}](5)(7)
            \Edge[color=green,lw=2.5pt,label={{\fontsize{8pt}{8pt}\selectfont \textcolor{green}{3}}}](7)(8)

            \Edge[color=orange,lw=0.5pt,label={{\fontsize{8pt}{8pt}\selectfont \textcolor{orange}{4}}}](2)(3)
            \Edge[color=orange,lw=0.5pt,label={{\fontsize{8pt}{8pt}\selectfont \textcolor{orange}{4}}}](2)(5)  
            \Edge[color=orange,lw=2.5pt,label={{\fontsize{8pt}{8pt}\selectfont \textcolor{orange}{4}}}](8)(9)  
            
        \end{tikzpicture}
    \caption{The first coloring of $E(S_3^2)$}
\end{subfigure}
\begin{subfigure}[b]{0.32\linewidth}
\begin{tikzpicture}[scale=0.6,inner sep=2pt, line width=0.2mm]    
             \draw (0,0) node(1) [circle,draw] {}
                  (2,0) node(2) [circle,draw] {}
                  (4,0) node(3) [circle,draw] {}
                  (6,0) node(4) [circle,draw] {}
                  (1,1.73) node(5) [circle,draw] {}
                  (5,1.73) node(6) [circle,draw] {}
                  (2,3.46) node(7) [circle,draw] {}
                  (4,3.46) node(8) [circle,draw] {}
                  (3,5.19) node(9) [circle,draw] {};
                  
            \Edge[color=red,lw=2.5pt,label={{\fontsize{8pt}{8pt}\selectfont \textcolor{red}{{\bf 1}}}}](3)(6)
            \Edge[color=red,lw=0.5pt,label={{\fontsize{8pt}{8pt}\selectfont \textcolor{red}{{\bf 1}}}}](5)(7)
            \Edge[color=red,lw=0.5pt,label={{\fontsize{8pt}{8pt}\selectfont \textcolor{red}{{\bf 1}}}}](7)(9)

            \Edge[color=cyan,lw=2.5pt,label={{\fontsize{8pt}{8pt}\selectfont \textcolor{cyan}{2}}}](6)(8)
            \Edge[color=cyan,lw=2.5pt,label={{\fontsize{8pt}{8pt}\selectfont \textcolor{cyan}{2}}}](4)(6)
            \Edge[color=cyan,lw=0.5pt,label={{\fontsize{8pt}{8pt}\selectfont \textcolor{cyan}{2}}}](2)(5)

            \Edge[color=green,lw=0.5pt,label={{\fontsize{8pt}{8pt}\selectfont \textcolor{green}{3}}}](1)(2)
            \Edge[color=green,lw=0.5pt,label={{\fontsize{8pt}{8pt}\selectfont \textcolor{green}{3}}}](2)(3)
            \Edge[color=green,lw=2.5pt,label={{\fontsize{8pt}{8pt}\selectfont \textcolor{green}{3}}}](7)(8)

            \Edge[color=orange,lw=0.5pt,label={{\fontsize{8pt}{8pt}\selectfont \textcolor{orange}{4}}}](3)(4)
            \Edge[color=orange,lw=0.5pt,label={{\fontsize{8pt}{8pt}\selectfont \textcolor{orange}{4}}}](1)(5)  
            \Edge[color=orange,lw=2.5pt,label={{\fontsize{8pt}{8pt}\selectfont \textcolor{orange}{4}}}](8)(9)

\end{tikzpicture}
        \caption{The second coloring of $E(S_3^2)$}
    \end{subfigure}
 \caption{Three subgraphs providing lower bounds.} 
 \label{ThesubgraphsofS3n}
\end{figure}

For $n=3$, it can be verified by computer that $\chi_i'(S_3^3)\ge 5$. Alternatively, note that $S_3^3$ consists of three copies of the graph $S_3^2$, and analyze possible injective edge colorings of $S_3^2$ using $4$ colors by starting with a coloring of a subgraph isomorphic to $H$ (see Figure~\ref{ThesubgraphsofS3n}(a) again). By symmetry, there are exactly two possibilities for the color of the edge that connects two triangles of $H$; see Figures~\ref{ThesubgraphsofS3n}(b) and (c), where edges of the corresponding subgraph isomorphic to $H$ are thick, and each of the two possibilities for the middle edge is considered by using color $1$ or $2$, respectively. Now, by analyzing possible extensions of the partial colorings using 4 colors on the five bold edges, one derives that each of the colorings uniquely extends to an injective edge coloring of $S_3^2$; these unique colorings are depicted in Figures~\ref{ThesubgraphsofS3n}(b) and (c). (While considering possibilities of extensions of these partial colorings of $S_3^2$ using 4 colors, one observes the colors in the three subgraphs isomorphic to $H$, which are formed by the three pairs of triangles. The colors are forced one by one resulting in the unique coloring in each of the two cases.)
Suppose that we could color $E(S_3^3)$ with only three colors. In $S_3^3$, two copies of $S_3^2$ are connected with an edge $e$, and the edge $e$ can be colored with the same color as one of its neighboring edges, but only in one of these two copies of $S_3^2$ (for otherwise there are three consecutive edges with the same color, which is impossible in an injective edge coloring). Hence, there is a copy of $S_3^2$, such that $e$ is colored with a color different from the colors of its neighboring edges. However, checking the colors of edges in Figures~\ref{ThesubgraphsofS3n}(b) and (c) that are in a triangle or incident with the triangle, one derives that no color in $[4]$ given to $e$ results in an injective edge coloring. Thus, $\chi_i'(S_3^3)\ge5$. Figure~\ref{IECSierpenskiofC3}(c) gives an injective edge coloring of $S_3^3$ with $5$ colors, which yields $\chi_i'(S_3^3)=5$.

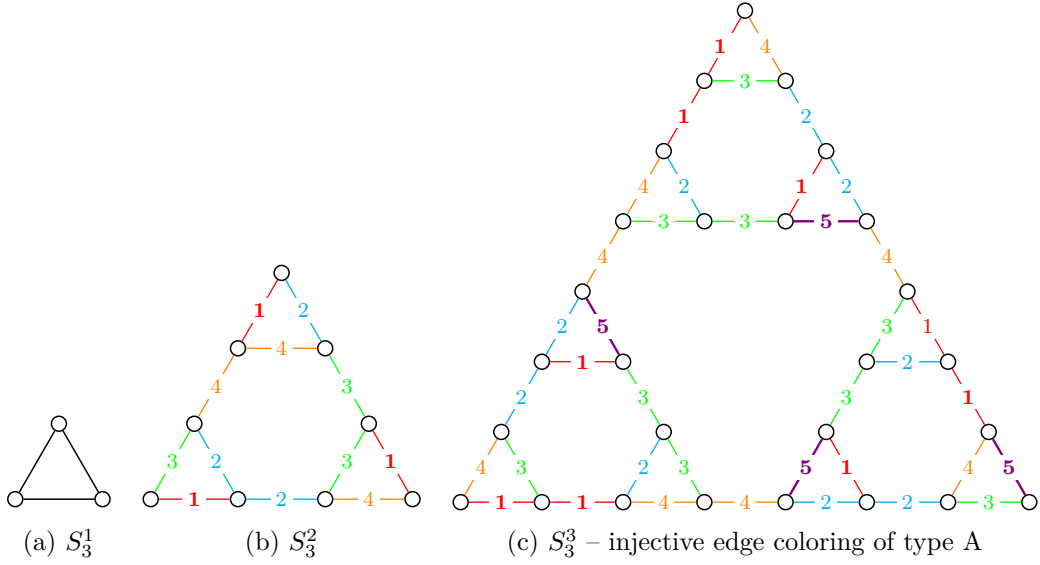
\begin{figure}[h!]
    \centering
        \begin{tabular}[b]{@{}ccc@{}}
        \begin{tikzpicture}[scale=0.58,inner sep=2pt, line width=0.2mm]
            \draw (-3,0) node(a) [circle,draw] {}
                  (-1,0) node(b) [circle,draw] {}
                  (-2,1.73) node(c) [circle,draw] {};
            \draw[] (a) to (b) to (c) to (a); 

            \draw(-2,-1) node(x) []{(a) $S_3^1$};
        \end{tikzpicture}
        
        &

\begin{tikzpicture}[scale=0.58,inner sep=2pt, line width=0.2mm]    
            \draw (0,0) node(1) [circle,draw] {}
                  (2,0) node(2) [circle,draw] {}
                  (4,0) node(3) [circle,draw] {}
                  (6,0) node(4) [circle,draw] {}
                  (1,1.73) node(5) [circle,draw] {}
                  (5,1.73) node(6) [circle,draw] {}
                  (2,3.46) node(7) [circle,draw] {}
                  (4,3.46) node(8) [circle,draw] {}
                  (3,5.19) node(9) [circle,draw] {};

            \Edge[color=red,lw=0.5pt,label={{\fontsize{8pt}{8pt}\selectfont \textcolor{red}{{\bf 1}}}}](1)(2)
            \Edge[color=red,lw=0.5pt,label={{\fontsize{8pt}{8pt}\selectfont \textcolor{red}{{\bf 1}}}}](4)(6)
            \Edge[color=red,lw=0.5pt,label={{\fontsize{8pt}{8pt}\selectfont \textcolor{red}{{\bf 1}}}}](7)(9)

            \Edge[color=cyan,lw=0.5pt,label={{\fontsize{8pt}{8pt}\selectfont \textcolor{cyan}{2}}}](5)(2)
            \Edge[color=cyan,lw=0.5pt,label={{\fontsize{8pt}{8pt}\selectfont \textcolor{cyan}{2}}}](2)(3)
            \Edge[color=cyan,lw=0.5pt,label={{\fontsize{8pt}{8pt}\selectfont \textcolor{cyan}{2}}}](9)(8)

            \Edge[color=green,lw=0.5pt,label={{\fontsize{8pt}{8pt}\selectfont \textcolor{green}{3}}}](1)(5)
            \Edge[color=green,lw=0.5pt,label={{\fontsize{8pt}{8pt}\selectfont \textcolor{green}{3}}}](3)(6)
            \Edge[color=green,lw=0.5pt,label={{\fontsize{8pt}{8pt}\selectfont \textcolor{green}{3}}}](6)(8)

            \Edge[color=orange,lw=0.5pt,label={{\fontsize{8pt}{8pt}\selectfont \textcolor{orange}{4}}}](3)(4)
            \Edge[color=orange,lw=0.5pt,label={{\fontsize{8pt}{8pt}\selectfont \textcolor{orange}{4}}}](5)(7)  
            \Edge[color=orange,lw=0.5pt,label={{\fontsize{8pt}{8pt}\selectfont \textcolor{orange}{4}}}](7)(8)  
            
            \draw(3,-1) node(x) []{(b) $S_3^2$};              
        \end{tikzpicture}

        &

    \begin{tikzpicture}[scale=0.54,inner sep=2pt, line width=0.2mm]
            \draw (7,0) node(1) [circle,draw] {}
                  (9,0) node(2) [circle,draw] {}
                  (11,0) node(3) [circle,draw] {}
                  (13,0) node(4) [circle,draw] {}
                  (8,1.73) node(5) [circle,draw] {}
                  (12,1.73) node(6) [circle,draw] {}
                  (9,3.46) node(7) [circle,draw] {}
                  (11,3.46) node(8) [circle,draw] {}
                  (10,5.19) node(9) [circle,draw] {};
            \draw (11,6.92) node(10) [circle,draw] {}
                  (13,6.92) node(11) [circle,draw] {}
                  (12,8.65) node(12) [circle,draw] {}
                  (13,10.38) node(13) [circle,draw] {}
                  (14,12.11) node(14) [circle,draw] {}
                  (15,10.38) node(15) [circle,draw] {}
                  (16,8.65) node(16) [circle,draw] {}
                  (15,6.92) node(17) [circle,draw] {}
                  (17,6.92) node(18) [circle,draw] {};
            \draw (18,5.19) node(19) [circle,draw] {}
                  (19,3.46) node(20) [circle,draw] {}
                  (17,3.46) node(21) [circle,draw] {}
                  (16,1.73) node(22) [circle,draw] {}
                  (20,1.73) node(23) [circle,draw] {}
                  (21,0) node(24) [circle,draw] {}
                  (19,0) node(25) [circle,draw] {}
                  (17,0) node(26) [circle,draw] {}
                  (15,0) node(27) [circle,draw] {};   

            \Edge[color=red,lw=0.5pt,label={{\fontsize{8pt}{8pt}\selectfont \textcolor{red}{{\bf 1}}}}](1)(2)
            \Edge[color=red,lw=0.5pt,label={{\fontsize{8pt}{8pt}\selectfont \textcolor{red}{{\bf 1}}}}](3)(2)
            \Edge[color=red,lw=0.5pt,label={{\fontsize{8pt}{8pt}\selectfont \textcolor{red}{{\bf 1}}}}](7)(8)
            \Edge[color=red,lw=0.5pt,label={{\fontsize{8pt}{8pt}\selectfont \textcolor{red}{{\bf 1}}}}](26)(22)
            \Edge[color=violet,lw=0.9pt,label={{\fontsize{8pt}{8pt}\selectfont \textcolor{violet}{\bf 5}}}](24)(23)
            \Edge[color=red,lw=0.5pt,label={{\fontsize{8pt}{8pt}\selectfont \textcolor{red}{{\bf 1}}}}](23)(20)
            \Edge[color=red,lw=0.5pt,label={{\fontsize{8pt}{8pt}\selectfont \textcolor{red}{{\bf 1 }}}}](16)(17)
            \Edge[color=red,lw=0.5pt,label={{\fontsize{8pt}{8pt}\selectfont \textcolor{red}{{\bf 1}}}}](12)(13)
            \Edge[color=red,lw=0.5pt,label={{\fontsize{8pt}{8pt}\selectfont \textcolor{red}{{\bf 1 }}}}](13)(14)

            \Edge[color=cyan,lw=0.5pt,label={{\fontsize{8pt}{8pt}\selectfont \textcolor{cyan}{2}}}](3)(6)
            \Edge[color=cyan,lw=0.5pt,label={{\fontsize{8pt}{8pt}\selectfont \textcolor{cyan}{2}}}](5)(7)
            \Edge[color=cyan,lw=0.5pt,label={{\fontsize{8pt}{8pt}\selectfont \textcolor{cyan}{2}}}](7)(9)
            \Edge[color=cyan,lw=0.5pt,label={{\fontsize{8pt}{8pt}\selectfont \textcolor{cyan}{2}}}](11)(12)
            \Edge[color=cyan,lw=0.5pt,label={{\fontsize{8pt}{8pt}\selectfont \textcolor{cyan}{2}}}](15)(16)
            \Edge[color=cyan,lw=0.5pt,label={{\fontsize{8pt}{8pt}\selectfont \textcolor{cyan}{2}}}](18)(16)
            \Edge[color=cyan,lw=0.5pt,label={{\fontsize{8pt}{8pt}\selectfont \textcolor{cyan}{2}}}](20)(21)
            \Edge[color=cyan,lw=0.5pt,label={{\fontsize{8pt}{8pt}\selectfont \textcolor{cyan}{2}}}](27)(26)
            \Edge[color=cyan,lw=0.5pt,label={{\fontsize{8pt}{8pt}\selectfont \textcolor{cyan}{2}}}](25)(26)

            \Edge[color=green,lw=0.5pt,label={{\fontsize{8pt}{8pt}\selectfont \textcolor{green}{3}}}](13)(15)
            \Edge[color=green,lw=0.5pt,label={{\fontsize{8pt}{8pt}\selectfont \textcolor{green}{3}}}](10)(11)
            \Edge[color=green,lw=0.5pt,label={{\fontsize{8pt}{8pt}\selectfont \textcolor{green}{3}}}](11)(17)
            \Edge[color=green,lw=0.5pt,label={{\fontsize{8pt}{8pt}\selectfont \textcolor{green}{3}}}](19)(21)
            \Edge[color=green,lw=0.5pt,label={{\fontsize{8pt}{8pt}\selectfont \textcolor{green}{3}}}](21)(22)
            \Edge[color=orange,lw=0.5pt,label={{\fontsize{8pt}{8pt}\selectfont \textcolor{orange}{4}}}](23)(25)
            \Edge[color=green,lw=0.5pt,label={{\fontsize{8pt}{8pt}\selectfont \textcolor{green}{3}}}](4)(6)
            \Edge[color=green,lw=0.5pt,label={{\fontsize{8pt}{8pt}\selectfont \textcolor{green}{3}}}](8)(6)
            \Edge[color=green,lw=0.5pt,label={{\fontsize{8pt}{8pt}\selectfont \textcolor{green}{3}}}](5)(2)

            \Edge[color=orange,lw=0.5pt,label={{\fontsize{8pt}{8pt}\selectfont \textcolor{orange}{4}}}](1)(5)
            \Edge[color=orange,lw=0.5pt,label={{\fontsize{8pt}{8pt}\selectfont \textcolor{orange}{4}}}](9)(10)
            \Edge[color=orange,lw=0.5pt,label={{\fontsize{8pt}{8pt}\selectfont \textcolor{orange}{4}}}](10)(12)
            \Edge[color=orange,lw=0.5pt,label={{\fontsize{8pt}{8pt}\selectfont \textcolor{orange}{4}}}](14)(15)
            \Edge[color=orange,lw=0.5pt,label={{\fontsize{8pt}{8pt}\selectfont \textcolor{orange}{4}}}](18)(19)
            \Edge[color=red,lw=0.5pt,label={{\fontsize{8pt}{8pt}\selectfont \textcolor{red}{1}}}](19)(20)
            \Edge[color=green,lw=0.5pt,label={{\fontsize{8pt}{8pt}\selectfont \textcolor{green}{3}}}](24)(25)
            \Edge[color=orange,lw=0.5pt,label={{\fontsize{8pt}{8pt}\selectfont \textcolor{orange}{4}}}](3)(4)
            \Edge[color=orange,lw=0.5pt,label={{\fontsize{8pt}{8pt}\selectfont \textcolor{orange}{4}}}](4)(27)
            
            \Edge[color=violet,lw=0.9pt,label={{\fontsize{8pt}{8pt}\selectfont \textcolor{violet}{\bf 5}}}](22)(27)
            \Edge[color=violet,lw=0.9pt,label={{\fontsize{8pt}{8pt}\selectfont \textcolor{violet}{\bf 5}}}](8)(9)
            \Edge[color=violet,lw=0.9pt,label={{\fontsize{8pt}{8pt}\selectfont \textcolor{violet}{\bf 5}}}](17)(18)  
            \draw(14,-1) node(x) []{(c) $S_3^3$ -- injective edge coloring of type A}; 
        \end{tikzpicture}
        \end{tabular}
    \caption{Sierpi\'nski graph over $K_3$ for $n=1,2$ and $3$.}
    \label{IECSierpenskiofC3}
\end{figure}


\begin{figure}[h]
    \centering
    \begin{tikzpicture}[scale=0.54,inner sep=2pt, line width=0.2mm]
            \draw (7,0) node(1) [circle,draw] {}
                  (9,0) node(2) [circle,draw] {}
                  (11,0) node(3) [circle,draw] {}
                  (13,0) node(4) [circle,draw] {}
                  (8,1.73) node(5) [circle,draw] {}
                  (12,1.73) node(6) [circle,draw] {}
                  (9,3.46) node(7) [circle,draw] {}
                  (11,3.46) node(8) [circle,draw] {}
                  (10,5.19) node(9) [circle,draw] {};
            \draw (11,6.92) node(10) [circle,draw] {}
                  (13,6.92) node(11) [circle,draw] {}
                  (12,8.65) node(12) [circle,draw] {}
                  (13,10.38) node(13) [circle,draw] {}
                  (14,12.11) node(14) [circle,draw] {}
                  (15,10.38) node(15) [circle,draw] {}
                  (16,8.65) node(16) [circle,draw] {}
                  (15,6.92) node(17) [circle,draw] {}
                  (17,6.92) node(18) [circle,draw] {};
            \draw (18,5.19) node(19) [circle,draw] {}
                  (19,3.46) node(20) [circle,draw] {}
                  (17,3.46) node(21) [circle,draw] {}
                  (16,1.73) node(22) [circle,draw] {}
                  (20,1.73) node(23) [circle,draw] {}
                  (21,0) node(24) [circle,draw] {}
                  (19,0) node(25) [circle,draw] {}
                  (17,0) node(26) [circle,draw] {}
                  (15,0) node(27) [circle,draw] {};   

            \Edge[color=green,lw=0.5pt,label={{\fontsize{8pt}{8pt}\selectfont \textcolor{green}{3}}}](1)(2)
            \Edge[color=red,lw=0.5pt,label={{\fontsize{8pt}{8pt}\selectfont \textcolor{red}{\bf 1}}}](3)(2)
            \Edge[color=red,lw=0.5pt,label={{\fontsize{8pt}{8pt}\selectfont \textcolor{red}{\bf 1}}}](7)(8)
            \Edge[color=red,lw=0.5pt,label={{\fontsize{8pt}{8pt}\selectfont \textcolor{red}{\bf 1}}}](26)(22)
            \Edge[color=violet,lw=0.9pt,label={{\fontsize{8pt}{8pt}\selectfont \textcolor{violet}{\bf 5}}}](24)(23)
            \Edge[color=red,lw=0.5pt,label={{\fontsize{8pt}{8pt}\selectfont \textcolor{red}{\bf 1}}}](23)(20)
            \Edge[color=red,lw=0.5pt,label={{\fontsize{8pt}{8pt}\selectfont \textcolor{red}{\bf 1}}}](16)(17)
            \Edge[color=red,lw=0.5pt,label={{\fontsize{8pt}{8pt}\selectfont \textcolor{red}{\bf 1}}}](12)(13)
            \Edge[color=red,lw=0.5pt,label={{\fontsize{8pt}{8pt}\selectfont \textcolor{red}{\bf 1}}}](13)(14)

            \Edge[color=cyan,lw=0.5pt,label={{\fontsize{8pt}{8pt}\selectfont \textcolor{cyan}{2}}}](3)(6)
            \Edge[color=cyan,lw=0.5pt,label={{\fontsize{8pt}{8pt}\selectfont \textcolor{cyan}{2}}}](5)(7)
            \Edge[color=cyan,lw=0.5pt,label={{\fontsize{8pt}{8pt}\selectfont \textcolor{cyan}{2}}}](7)(9)
            \Edge[color=cyan,lw=0.5pt,label={{\fontsize{8pt}{8pt}\selectfont \textcolor{cyan}{2}}}](11)(12)
            \Edge[color=cyan,lw=0.5pt,label={{\fontsize{8pt}{8pt}\selectfont \textcolor{cyan}{2}}}](15)(16)
            \Edge[color=cyan,lw=0.5pt,label={{\fontsize{8pt}{8pt}\selectfont \textcolor{cyan}{2}}}](18)(16)
            \Edge[color=cyan,lw=0.5pt,label={{\fontsize{8pt}{8pt}\selectfont \textcolor{cyan}{2}}}](20)(21)
            \Edge[color=cyan,lw=0.5pt,label={{\fontsize{8pt}{8pt}\selectfont \textcolor{cyan}{2}}}](27)(26)
            \Edge[color=cyan,lw=0.5pt,label={{\fontsize{8pt}{8pt}\selectfont \textcolor{cyan}{2}}}](25)(26)

            \Edge[color=green,lw=0.5pt,label={{\fontsize{8pt}{8pt}\selectfont \textcolor{green}{3}}}](13)(15)
            \Edge[color=green,lw=0.5pt,label={{\fontsize{8pt}{8pt}\selectfont \textcolor{green}{3}}}](10)(11)
            \Edge[color=green,lw=0.5pt,label={{\fontsize{8pt}{8pt}\selectfont \textcolor{green}{3}}}](11)(17)
            \Edge[color=green,lw=0.5pt,label={{\fontsize{8pt}{8pt}\selectfont \textcolor{green}{3}}}](19)(21)
            \Edge[color=green,lw=0.5pt,label={{\fontsize{8pt}{8pt}\selectfont \textcolor{green}{3}}}](21)(22)
            \Edge[color=orange,lw=0.5pt,label={{\fontsize{8pt}{8pt}\selectfont \textcolor{orange}{4}}}](23)(25)
            \Edge[color=green,lw=0.5pt,label={{\fontsize{8pt}{8pt}\selectfont \textcolor{green}{3}}}](4)(6)
            \Edge[color=green,lw=0.5pt,label={{\fontsize{8pt}{8pt}\selectfont \textcolor{green}{3}}}](8)(6)
            \Edge[color=orange,lw=0.5pt,label={{\fontsize{8pt}{8pt}\selectfont \textcolor{orange}{4}}}](5)(2)

            \Edge[color=violet,lw=0.9pt,label={{\fontsize{8pt}{8pt}\selectfont \textcolor{violet}{\bf 5}}}](1)(5)
            \Edge[color=orange,lw=0.5pt,label={{\fontsize{8pt}{8pt}\selectfont \textcolor{orange}{4}}}](9)(10)
            \Edge[color=orange,lw=0.5pt,label={{\fontsize{8pt}{8pt}\selectfont \textcolor{orange}{4}}}](10)(12)
            \Edge[color=orange,lw=0.5pt,label={{\fontsize{8pt}{8pt}\selectfont \textcolor{orange}{4}}}](14)(15)
            \Edge[color=orange,lw=0.5pt,label={{\fontsize{8pt}{8pt}\selectfont \textcolor{orange}{4}}}](18)(19)
            \Edge[color=red,lw=0.5pt,label={{\fontsize{8pt}{8pt}\selectfont \textcolor{red}{\bf 1}}}](19)(20)
            \Edge[color=green,lw=0.5pt,label={{\fontsize{8pt}{8pt}\selectfont \textcolor{green}{3}}}](24)(25)
            \Edge[color=red,lw=0.5pt,label={{\fontsize{8pt}{8pt}\selectfont \textcolor{red}{\bf 1}}}](3)(4)
            \Edge[color=orange,lw=0.5pt,label={{\fontsize{8pt}{8pt}\selectfont \textcolor{orange}{4}}}](4)(27)
            
            \Edge[color=violet,lw=0.9pt,label={{\fontsize{8pt}{8pt}\selectfont \textcolor{violet}{\bf 5}}}](22)(27)
            \Edge[color=violet,lw=0.9pt,label={{\fontsize{8pt}{8pt}\selectfont \textcolor{violet}{\bf 5}}}](8)(9)
            \Edge[color=violet,lw=0.9pt,label={{\fontsize{8pt}{8pt}\selectfont \textcolor{violet}{\bf 5}}}](17)(18)  
            \draw(14,-1) node(x) []{$S_3^3$ -- injective edge coloring of type B}; 
        \end{tikzpicture}
        
    \caption{Type B.}
    \label{fig:S3typeB}
\end{figure}

Now, we extend the above coloring to an injective edge coloring of $S_3^n$ for $n\geq 4$.
In fact, we obtain two types of injective edge colorings of $S_3^n$ for each $n\ge 3$. 
For $n=3$ these two coloring are presented in two figures: Figure~\ref{IECSierpenskiofC3}(c) shows the coloring called Type A, and Figure~\ref{fig:S3typeB} the coloring called Type B. The reader can verify that both colorings are injective edge colorings of $S_3^3$ using $5$ colors. Important properties of these colorings are as follows:
\begin{enumerate}[label=\roman*.,noitemsep]
\item the top extreme vertex has incident edges colored with colors $1$ and $4$;
\item the bottom-right extreme vertex has incident edges colored with colors $3$ and $5$ and the edge with color $5$ has no adjacent edges with color $5$; 
\item for the bottom-left extreme vertex: in Type A coloring its incident edges are colored with colors $1$ and $4$, while in Type B coloring its incident edges are colored with colors $3$ and $5$ and the edge with color $5$ has no adjacent edges with color $5$. 
\end{enumerate}

\begin{figure}[H]
    \centering
    \begin{tikzpicture}[scale=0.58,inner sep=2pt, line width=0.2mm]
        \tikzstyle{triangle}=[draw, shape=regular polygon, regular polygon sides=3,draw,thick,inner sep=0pt,minimum
size=4cm], 

\draw (-3.75,3.35) node(a) [] {};
\draw (-2.89,4.9) node(b) [] {};
\draw (3.75,3.35) node(c) [] {};
\draw (2.89,4.9) node(d) [] {};
\draw (-0.85,-1.72) node(e) [] {};
\draw (0.85,-1.72) node(f) [] {};

        \node [triangle](1) at (-3.7,0) {type A};
          \node [triangle](1) at (3.7,0) {type A};
            \node [triangle](1) at (0,6.5) {type B};
        
           \Edge[color=violet,lw=1.2pt,label={{\fontsize{6pt}{8pt}\selectfont \textcolor{violet}{\bf 5}}}](a)(b)
           
  \Edge[color=violet,lw=1.2pt,label={{\fontsize{6pt}{8pt}\selectfont \textcolor{violet}{\bf 5}}}](c)(d)  

  \Edge[color=violet,lw=1.2pt,label={{\fontsize{6pt}{8pt}\selectfont \textcolor{violet}{\bf 5}}}](e)(f)  

  \draw(-3.7,1) node(x) []{$S_3^{n-1}$};    
    \draw(3.7,1) node(y) []{$S_3^{n-1}$}; 
    \draw(0,7.5) node(z) []{$S_3^{n-1}$}; 
           \end{tikzpicture}
\caption{Coloring of Type A of Sierpi\'nski graph $S_3^n$.}
\label{fig:TypeA}
\end{figure}

\begin{figure}[h!]
    \centering
    \begin{tikzpicture}[scale=0.58,inner sep=2pt, line width=0.2mm]
        \tikzstyle{triangle}=[draw, shape=regular polygon, regular polygon sides=3,draw,thick,inner sep=0pt,minimum
size=4cm], 

\draw (-3.75,3.35) node(a) [] {};
\draw (-2.89,4.9) node(b) [] {};
\draw (3.75,3.35) node(c) [] {};
\draw (2.89,4.9) node(d) [] {};
\draw (-0.85,-1.72) node(e) [] {};
\draw (0.85,-1.72) node(f) [] {};

        \node [triangle](1) at (-3.7,0) {type B};
          \node [triangle](1) at (3.7,0) {type A};
            \node [triangle](1) at (0,6.5) {type B};
        
           \Edge[color=violet,lw=1.2pt,label={{\fontsize{9pt}{8pt}\selectfont \textcolor{violet}{\bf 5}}}](a)(b)
           
  \Edge[color=violet,lw=1.2pt,label={{\fontsize{9pt}{8pt}\selectfont \textcolor{violet}{\bf 5}}}](c)(d)  

  \Edge[color=violet,lw=1.2pt,label={{\fontsize{9pt}{8pt}\selectfont \textcolor{violet}{\bf 5}}}](e)(f)  

  \draw(-3.7,1) node(x) []{$S_3^{n-1}$};    
    \draw(3.7,1) node(y) []{$S_3^{n-1}$}; 
    \draw(0,7.5) node(z) []{$S_3^{n-1}$}; 
           \end{tikzpicture}
\caption{Coloring of Type B of Sierpi\'nski graph $S_3^n$.}
\label{fig:TypeB}
\end{figure}

For $n\ge 4$, we also obtain two injective edge colorings of $S_3^n$, again called Type A and Type B, that have the same three properties as described for $S_3^3$. We achieve this as follows. Note that $S_3^n$ can be obtained from three copies of $S_3^{n-1}$, as shown in Figure~\ref{fig:TypeA}. The figure also indicates which colorings of each of the copies of $S_3^{n-1}$ is used in each part (either coloring of Type A or Type B). The three edges between the copies of $S_3^{n-1}$ are colored by $5$. One can easily verify that the resulting coloring is also an injective edge coloring of $S_3^{n}$, where we assume by induction that the colorings of $S_3^{n-1}$ are injective edge colorings of Types A or B with specific properties, i, ii and iii, described above. In addition, note that the coloring indicated in Figure~\ref{fig:TypeA} is an injective edge coloring of $S_3^n$ of Type A, as defined in i, ii, and iii. 
Now, Figure~\ref{fig:TypeB} presents a similar construction, which also gives an injective edge coloring of $S_3^n$, this time of Type B. By induction on the dimension $n$ of the Sierpi\'{n}ski graph $S_3^n$, we obtain both types of injective edge colorings of $S_3^{n}$, which completes the proof.\end{proof}

The authors of~\cite{bksy} obtained the exact formula for the injective chromatic number of standard Sierpi\' nski graphs ($\chi_i(S_{K_p}^n)=p$ for any $p\ge 3$ and any $n$). The proof of the formula was general in the sense that the order $p\ge 3$ of the complete graph $K_p$ did not play any particular role. We think that determining the injective chromatic indices of the graphs $S_{K_p}^n$ will be much more demanding. In addition, we think that the proofs for different orders $p$ may be different, and may be slightly easier for some smaller values of $p$ such as $4$ or $5$. We pose this as the following problem for which even partial solutions may be interesting. 

\begin{problem}
Determine $\chi_i'(S_{K_p}^n)$ for $p\ge 4$ and $n\ge 2$.     
\end{problem}

Next, we present a general result on the injective chromatic index in generalized Sierpi\'{n}ski graphs, which is in the same spirit as some results in Section~\ref{sec:inj}. Notably, in Theorem~\ref{thm:SG2} we proved that $\chi_i(S_G^2)\le\chi_i(S_G^n)\le \chi_i(S_G^2)+1$ holds for any graph $G$ and any $n\ge 3$, and then in Theorem~\ref{thm:SG2suff} gave a sufficient condition for attaining the equality $\chi_i(S_G^n)=\chi_i(S_G^2)$. The main difference is that in the vertex version of injective coloring the bounds were with respect to $S_G^2$, while in the edge version the bounds are with respect to $S_G^3$. In addition, we need to restrict to triangle-free graphs.

\begin{theorem}
\label{thm:triangle-free}
For any triangle-free graph $G$ and any positive integer $n\ge 3$,
	$$\chi_i'(S_G^3)\le \chi_i'(S_G^n)\le \chi_i'(S_G^3)+1.$$
In addition, $\chi_i'(S_G^n)=\chi_i'(S_G^3)$ for all $n\ge 3$ if
\begin{enumerate}[label=(\roman*),noitemsep]
    \item there exists a $\chi_i'$-coloring of $S_G^3$, where $\chi_i'(S_G^3)=k$, such that every vertex in $S_G^3$ incident to an edge with color $k$ is at distance at least $3$ from every extreme vertex $(u,u,u)$, $u\in V(G)$, of $S_G^3$; or
    \item $G$ is bipartite with the bipartition $V(G)=A\cup B$ and there exists a $\chi_i'$-coloring of $S_G^3$, where $\chi_i'(S_G^3)=k$, such that for every $u\in A$ the edges incident with the extreme vertex $(u,u,u)$ receive color $k$, and every vertex in $S_G^3$ incident to an edge with color $k$ is at distance at least $2$ from every extreme vertex $(v,v,v)$, where $v\in B$.
\end{enumerate}
\label{thm:SG3}
\end{theorem}
\begin{proof}
Clearly, $\chi_i(S_G^n)\geq \chi_i(S_G^3)$, since $S_G^3$ is a subgraph of $S_G^n$.

Let $f:E(S_G^3)\rightarrow [k]$ be a $\chi_i'$-coloring of $S_G^3$. For any $n\ge 4$ we will introduce a coloring $f_1:E(S_G^n)\rightarrow [k+1]$ for which we will prove it is an injective edge coloring of $S_G^n$. Note that $V(S_G^n)$ can be partitioned into $|V(G)|^{n-3}$ subsets each of which induces the graph $S_G^3$.  First, we assign the color $k+1$ to the edges between different copies of $S_G^3$ in $S_G^n$. That is, for any edge $(u_1,\ldots,u_{n-3},v,v,v)(v_1,\ldots,v_{n-3},u,u,u)$, where $u_i,v_i\in V(G)$ for all $i\in [n-3]$ and $uv\in E(G)$, we assign $$f_1((u_1,\ldots,u_{n-3},v,v,v)(v_1,\ldots,v_{n-3},u,u,u))=k+1$$ 
(note that there exists $d\in [n-4]$ such that $u_i=v_i$ for all $i<d$, $u_d=u$, $v_d=v$, and $u_i=v$, $v_i=u$ for all $i\in\{d+1,\ldots,n-4\}$).
Concerning the edges whose endvertices belong to one and the same copy of $S_G^3$, we color them in the same way as they are colored by $f$ with some exceptions that we specify next. For this purpose, for each edge $uv\in E(G)$ we arbitrarily choose one endvertex, either $u$ or $v$, and we call it the chosen vertex of the edge $uv$. Now, consider a copy of $S_G^3$ in $S_G^n$ and let $(w_1,\ldots,w_{n-3})\in V(G)^{n-3}$ be the vertex to which it corresponds; that is, consider the subgraph $S_{(w_1,\ldots,w_{n-3})}$ of $S_G^n$ induced by $$\{(w_1,\ldots,w_{n-3},x,y,z):\, \textrm{ where }(x,y,z)\in V(G)^3\}.$$
If a vertex in $S_{(w_1,\ldots,w_{n-3})}$ is an extreme vertex of that copy of $S_G^3$, say a vertex $(w_1,\ldots,w_{n-3},x,x,x)$, which has a neighbor $(u_1,\ldots,u_{n-3},y,y,y)$ in another copy of $S_G^3$, namely $S_{(u_1,\ldots,u_{n-3})}$, and $x$ is the chosen endvertex of the edge $xy\in E(G)$, then $f_1$ colors all edges of $S_{(w_1,\ldots,w_{n-3})}$ incident with $(w_1,\ldots,w_{n-3},x,x,x)$ by color $k+1$. On the other hand, $f_1$ colors all other edges between vertices in $S_{(w_1,\ldots,w_{n-3})}$ with the same color as $f$ colors the corresponding edges in $S_G^3$. That is, $f_1(w_1,\ldots,w_{n-3},u,v,w)(w_1,\ldots,w_{n-3},x,y,z))=f((u,v,w)(x,y,z))$ for all such edges. 

To prove that $f_1$ is indeed an injective edge coloring of $S_G^n$, first consider the edges colored by color $k+1$. Note that the distance between any two extreme vertices in $S_G^3$ is at least $7$. Hence, if two edges that are in the same copy of $S_G^3$ receive the color $k+1$ by $f_1$, then either they have a common endvertex, which is an extreme vertex of that copy, or their distance is at least $5$ (since they are incident to distinct extreme vertices of that copy of $S_G^3$ in $S_G^n$). In the former case, they do not have a common edge, since $G$ is triangle-free. In the latter case, since their distance is at least $5$, they also do not have a common edge.  
Next, the distance between two edges colored with color $k+1$ that are from distinct copies of $S_G^3$ is at least $7$. Finally, the distance from an edge $e$ with color $k+1$ that lies between two distinct copies of $S_G^3$ in $S_G^n$ to any other edge with color $k+1$ is at least $6$, unless an edge with color $k+1$ is adjacent to $e$ (in which case they clearly do not have any common edges). 

Second, concerning the edges $e$ with $f_1(e)\le k$, we distinguish two cases. If two edges $e$ and $e'$ with $f_1(e)=f_1(e')$ belong to the same copy of $S_G^3$ in $S_G^n$, then since $f_1(e)=f(e)=f(e')=f_1(e')$, they do not have a common edge, since $f$ is an injective edge coloring of $S_G^3$. On the other hand, if $e$ and $e'$ are edges from distinct copies of $S_G^3$, and $f_1(e)=f_1(e')\le k$, then by the construction of $f_1$, they also cannot have a common edge. Indeed, on any path between $e$ and $e'$ there lie two edges colored with color $k+1$, hence $e$ and $e'$ do not have a common edge. Thus, $f_1$ is an injective edge coloring of $S_G^n$, which uses $k+1$ colors, so the upper bound $\chi_i'(S_G^n)\le \chi_i'(S_G^3)+1$ is proved. 

For the proof of {\it (i)}, assume that there exists a $\chi_i'$-coloring of $S_G^3$, where $\chi_i'(S_G^3)=k$, such that every vertex in $S_G^3$ incident to an edge with color $k$ is at distance at least $3$ from every extreme vertex $(u,u,u)$, $u\in V(G)$, of $S_G^3$. Then $f_1$ can be defined in almost the same way as in the previous (general) case, except that all edges that $f_1$ colored to $k+1$ in the previous case are recolored to $k$. Indeed, the condition that the distance from every extreme vertex $(u,u,u)$, $u\in V(G)$, of $S_G^3$, to a vertex incident with an edge with color $k$ is at least $3$ ensures that no two edges with color $k$ will have a common edge.

The proof of {\it (ii)} is straightforward. We define the coloring $f_1$ by 
$$f_1((w_1,\ldots,w_{n-3},u,v,w)(w_1,\ldots,w_{n-3},x,y,z))=f((u,v,w)(x,y,z))$$ for every edge in $S_G^n$ whose endvertices lie in the same copy of $S_G^3$ of $S_G^n$, where $f$ is a coloring of $S_G^3$ that satisfies the properties in the assumption of the statement. That is, in each copy of $S_G^3$ we use the coloring $f$. Now, to each edge between distinct copies of $S_G^3$ in $S_G^n$ let $f_1$ assign color $k$. Note that an edge between two vertices $(u_1,\ldots, u_n),(v_1,\ldots,v_n)\in V(S_G^n)$ that belong to distinct copies of $S_G^3$ exists only if $u_{n-2}=u_{n-1}=u_{n}=u$ and $v_{n-2}=v_{n-1}=v_{n}=v$ and $uv\in E(G)$. Since $G$ is bipartite, we may assume that $u\in A$ and $v\in B$. Note that all edges incident with $(u_1,\ldots, u_n)$ are colored with color $k$ by $f_1$, and since $G$ is triangle-free they do share a common edge. On the other hand, the distance from $(v_1,\ldots,v_n)$ to any vertex in the same copy of $S_G^3$, which is incident to an edge with color $k$, is at least $2$. Therefore, the edge $(u_1,\ldots, u_n)(v_1,\ldots,v_n)$ has no common edge to any edge with color $k$. 
The proof is complete.  
\end{proof}

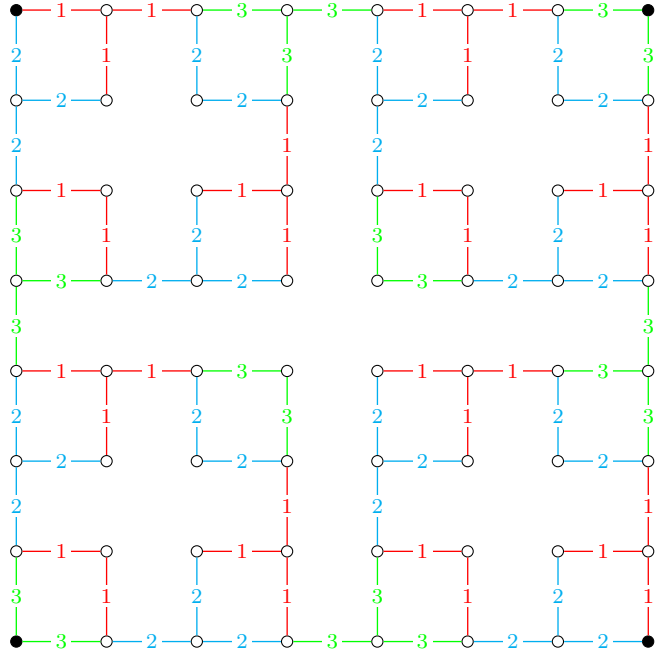
\begin{figure}[t]
\begin{center}
\begin{tikzpicture}[scale=1,inner sep=1.5pt]
\draw [] (0,0) node(1) [circle, draw, fill] {}
         (1,0) node(2) [circle, draw] {}
         (2,0) node(3) [circle, draw] {}
         (3,0) node(4) [circle, draw] {}
         (4,0) node(5) [circle, draw] {}
         (5,0) node(6) [circle, draw] {}
         (6,0) node(7) [circle, draw] {}
         (7,0) node(8) [circle, draw, fill] {}
         
         (0,1) node(9) [circle, draw] {}
         (1,1) node(10) [circle, draw] {}
         (2,1) node(11) [circle, draw] {}
         (3,1) node(12) [circle, draw] {}
         (4,1) node(13) [circle, draw] {}
         (5,1) node(14) [circle, draw] {}
         (6,1) node(15) [circle, draw] {}
         (7,1) node(16) [circle, draw] {}
         
         (0,2) node(17) [circle, draw] {}
         (1,2) node(18) [circle, draw] {}
         (2,2) node(19) [circle, draw] {}
         (3,2) node(20) [circle, draw] {}
         (4,2) node(21) [circle, draw] {}
         (5,2) node(22) [circle, draw] {}
         (6,2) node(23) [circle, draw] {}
         (7,2) node(24) [circle, draw] {}  

         (0,3) node(25) [circle, draw] {}
         (1,3) node(26) [circle, draw] {}
         (2,3) node(27) [circle, draw] {}
         (3,3) node(28) [circle, draw] {}
         (4,3) node(29) [circle, draw] {}
         (5,3) node(30) [circle, draw] {}
         (6,3) node(31) [circle, draw] {}
         (7,3) node(32) [circle, draw] {} 

         (0,4) node(33) [circle, draw] {}
         (1,4) node(34) [circle, draw] {}
         (2,4) node(35) [circle, draw] {}
         (3,4) node(36) [circle, draw] {}
         (4,4) node(37) [circle, draw] {}
         (5,4) node(38) [circle, draw] {}
         (6,4) node(39) [circle, draw] {}
         (7,4) node(40) [circle, draw] {} 

         (0,5) node(41) [circle, draw] {}
         (1,5) node(42) [circle, draw] {}
         (2,5) node(43) [circle, draw] {}
         (3,5) node(44) [circle, draw] {}
         (4,5) node(45) [circle, draw] {}
         (5,5) node(46) [circle, draw] {}
         (6,5) node(47) [circle, draw] {}
         (7,5) node(48) [circle, draw] {} 

         (0,6) node(49) [circle, draw] {}
         (1,6) node(50) [circle, draw] {}
         (2,6) node(51) [circle, draw] {}
         (3,6) node(52) [circle, draw] {}
         (4,6) node(53) [circle, draw] {}
         (5,6) node(54) [circle, draw] {}
         (6,6) node(55) [circle, draw] {}
         (7,6) node(56) [circle, draw] {} 

         (0,7) node(57) [circle, draw, fill] {}
         (1,7) node(58) [circle, draw] {}
         (2,7) node(59) [circle, draw] {}
         (3,7) node(60) [circle, draw] {}
         (4,7) node(61) [circle, draw] {}
         (5,7) node(62) [circle, draw] {}
         (6,7) node(63) [circle, draw] {}
         (7,7) node(64) [circle, draw, fill] {} ;

    \Edge[color=green,lw=0.5pt,label={{\fontsize{8pt}{8pt}\selectfont \textcolor{green}{3}}}](1)(2)  
    \Edge[color=cyan,lw=0.5pt,label={{\fontsize{8pt}{8pt}\selectfont \textcolor{cyan}{2}}}](2)(3)
    \Edge[color=cyan,lw=0.5pt,label={{\fontsize{8pt}{8pt}\selectfont \textcolor{cyan}{2}}}](3)(4)
    \Edge[color=green,lw=0.5pt,label={{\fontsize{8pt}{8pt}\selectfont \textcolor{green}{3}}}](4)(5)
    \Edge[color=green,lw=0.5pt,label={{\fontsize{8pt}{8pt}\selectfont \textcolor{green}{3}}}](5)(6)
    \Edge[color=cyan,lw=0.5pt,label={{\fontsize{8pt}{8pt}\selectfont \textcolor{cyan}{2}}}](6)(7)
    \Edge[color=cyan,lw=0.5pt,label={{\fontsize{8pt}{8pt}\selectfont \textcolor{cyan}{2}}}](7)(8)

    \Edge[color=red,lw=0.5pt,label={{\fontsize{8pt}{8pt}\selectfont \textcolor{red}{1}}}](9)(10)
    \Edge[color=red,lw=0.5pt,label={{\fontsize{8pt}{8pt}\selectfont \textcolor{red}{1}}}](11)(12)
    \Edge[color=red,lw=0.5pt,label={{\fontsize{8pt}{8pt}\selectfont \textcolor{red}{1}}}](13)(14)
    \Edge[color=red,lw=0.5pt,label={{\fontsize{8pt}{8pt}\selectfont \textcolor{red}{1}}}](15)(16)

    \Edge[color=cyan,lw=0.5pt,label={{\fontsize{8pt}{8pt}\selectfont \textcolor{cyan}{2}}}](17)(18)
    \Edge[color=cyan,lw=0.5pt,label={{\fontsize{8pt}{8pt}\selectfont \textcolor{cyan}{2}}}](19)(20)
    \Edge[color=cyan,lw=0.5pt,label={{\fontsize{8pt}{8pt}\selectfont \textcolor{cyan}{2}}}](21)(22)
    \Edge[color=cyan,lw=0.5pt,label={{\fontsize{8pt}{8pt}\selectfont \textcolor{cyan}{2}}}](23)(24) 
    
    \Edge[color=red,lw=0.5pt,label={{\fontsize{8pt}{8pt}\selectfont \textcolor{red}{1}}}](25)(26)
    \Edge[color=red,lw=0.5pt,label={{\fontsize{8pt}{8pt}\selectfont \textcolor{red}{1}}}](26)(27)
    \Edge[color=green,lw=0.5pt,label={{\fontsize{8pt}{8pt}\selectfont \textcolor{green}{3}}}](27)(28)
    \Edge[color=red,lw=0.5pt,label={{\fontsize{8pt}{8pt}\selectfont \textcolor{red}{1}}}](29)(30)
    \Edge[color=red,lw=0.5pt,label={{\fontsize{8pt}{8pt}\selectfont \textcolor{red}{1}}}](30)(31) 
    \Edge[color=green,lw=0.5pt,label={{\fontsize{8pt}{8pt}\selectfont \textcolor{green}{3}}}](31)(32)

    \Edge[color=green,lw=0.5pt,label={{\fontsize{8pt}{8pt}\selectfont \textcolor{green}{3}}}](33)(34)
    \Edge[color=cyan,lw=0.5pt,label={{\fontsize{8pt}{8pt}\selectfont \textcolor{cyan}{2}}}](34)(35)
    \Edge[color=cyan,lw=0.5pt,label={{\fontsize{8pt}{8pt}\selectfont \textcolor{cyan}{2}}}](35)(36)
    \Edge[color=green,lw=0.5pt,label={{\fontsize{8pt}{8pt}\selectfont \textcolor{green}{3}}}](37)(38) 
    \Edge[color=cyan,lw=0.5pt,label={{\fontsize{8pt}{8pt}\selectfont \textcolor{cyan}{2}}}](38)(39)
    \Edge[color=cyan,lw=0.5pt,label={{\fontsize{8pt}{8pt}\selectfont \textcolor{cyan}{2}}}](39)(40) 

    \Edge[color=red,lw=0.5pt,label={{\fontsize{8pt}{8pt}\selectfont \textcolor{red}{1}}}](41)(42)
    \Edge[color=red,lw=0.5pt,label={{\fontsize{8pt}{8pt}\selectfont \textcolor{red}{1}}}](43)(44)
    \Edge[color=red,lw=0.5pt,label={{\fontsize{8pt}{8pt}\selectfont \textcolor{red}{1}}}](45)(46)
    \Edge[color=red,lw=0.5pt,label={{\fontsize{8pt}{8pt}\selectfont \textcolor{red}{1}}}](47)(48) 
    
    \Edge[color=cyan,lw=0.5pt,label={{\fontsize{8pt}{8pt}\selectfont \textcolor{cyan}{2}}}](49)(50)
    \Edge[color=cyan,lw=0.5pt,label={{\fontsize{8pt}{8pt}\selectfont \textcolor{cyan}{2}}}](51)(52)
    \Edge[color=cyan,lw=0.5pt,label={{\fontsize{8pt}{8pt}\selectfont \textcolor{cyan}{2}}}](53)(54)
    \Edge[color=cyan,lw=0.5pt,label={{\fontsize{8pt}{8pt}\selectfont \textcolor{cyan}{2}}}](55)(56) 

    \Edge[color=red,lw=0.5pt,label={{\fontsize{8pt}{8pt}\selectfont \textcolor{red}{1}}}](57)(58) 
    \Edge[color=red,lw=0.5pt,label={{\fontsize{8pt}{8pt}\selectfont \textcolor{red}{1}}}](58)(59)
    \Edge[color=green,lw=0.5pt,label={{\fontsize{8pt}{8pt}\selectfont \textcolor{green}{3}}}](59)(60)
    \Edge[color=green,lw=0.5pt,label={{\fontsize{8pt}{8pt}\selectfont \textcolor{green}{3}}}](60)(61)
    \Edge[color=red,lw=0.5pt,label={{\fontsize{8pt}{8pt}\selectfont \textcolor{red}{1}}}](61)(62)
    \Edge[color=red,lw=0.5pt,label={{\fontsize{8pt}{8pt}\selectfont \textcolor{red}{1}}}](62)(63)
    \Edge[color=green,lw=0.5pt,label={{\fontsize{8pt}{8pt}\selectfont \textcolor{green}{3}}}](63)(64)

    \Edge[color=green,lw=0.5pt,label={{\fontsize{8pt}{8pt}\selectfont \textcolor{green}{3}}}](1)(9)  
    \Edge[color=cyan,lw=0.5pt,label={{\fontsize{8pt}{8pt}\selectfont \textcolor{cyan}{2}}}](9)(17)
    \Edge[color=cyan,lw=0.5pt,label={{\fontsize{8pt}{8pt}\selectfont \textcolor{cyan}{2}}}](17)(25)
    \Edge[color=green,lw=0.5pt,label={{\fontsize{8pt}{8pt}\selectfont \textcolor{green}{3}}}](25)(33)
    \Edge[color=green,lw=0.5pt,label={{\fontsize{8pt}{8pt}\selectfont \textcolor{green}{3}}}](33)(41)
    \Edge[color=cyan,lw=0.5pt,label={{\fontsize{8pt}{8pt}\selectfont \textcolor{cyan}{2}}}](41)(49)
    \Edge[color=cyan,lw=0.5pt,label={{\fontsize{8pt}{8pt}\selectfont \textcolor{cyan}{2}}}](49)(57)

    \Edge[color=red,lw=0.5pt,label={{\fontsize{8pt}{8pt}\selectfont \textcolor{red}{1}}}](2)(10)
    \Edge[color=red,lw=0.5pt,label={{\fontsize{8pt}{8pt}\selectfont \textcolor{red}{1}}}](18)(26)
    \Edge[color=red,lw=0.5pt,label={{\fontsize{8pt}{8pt}\selectfont \textcolor{red}{1}}}](34)(42)
    \Edge[color=red,lw=0.5pt,label={{\fontsize{8pt}{8pt}\selectfont \textcolor{red}{1}}}](50)(58) 

    \Edge[color=cyan,lw=0.5pt,label={{\fontsize{8pt}{8pt}\selectfont \textcolor{cyan}{2}}}](3)(11)
    \Edge[color=cyan,lw=0.5pt,label={{\fontsize{8pt}{8pt}\selectfont \textcolor{cyan}{2}}}](19)(27)
    \Edge[color=cyan,lw=0.5pt,label={{\fontsize{8pt}{8pt}\selectfont \textcolor{cyan}{2}}}](35)(43)
    \Edge[color=cyan,lw=0.5pt,label={{\fontsize{8pt}{8pt}\selectfont \textcolor{cyan}{2}}}](51)(59) 

    \Edge[color=red,lw=0.5pt,label={{\fontsize{8pt}{8pt}\selectfont \textcolor{red}{1}}}](4)(12)  
    \Edge[color=red,lw=0.5pt,label={{\fontsize{8pt}{8pt}\selectfont \textcolor{red}{1}}}](12)(20)
    \Edge[color=green,lw=0.5pt,label={{\fontsize{8pt}{8pt}\selectfont \textcolor{green}{3}}}](20)(28)
    \Edge[color=red,lw=0.5pt,label={{\fontsize{8pt}{8pt}\selectfont \textcolor{red}{1}}}](36)(44)
    \Edge[color=red,lw=0.5pt,label={{\fontsize{8pt}{8pt}\selectfont \textcolor{red}{1}}}](44)(52)
    \Edge[color=green,lw=0.5pt,label={{\fontsize{8pt}{8pt}\selectfont \textcolor{green}{3}}}](52)(60)

    \Edge[color=green,lw=0.5pt,label={{\fontsize{8pt}{8pt}\selectfont \textcolor{green}{3}}}](5)(13)  
    \Edge[color=cyan,lw=0.5pt,label={{\fontsize{8pt}{8pt}\selectfont \textcolor{cyan}{2}}}](13)(21)
    \Edge[color=cyan,lw=0.5pt,label={{\fontsize{8pt}{8pt}\selectfont \textcolor{cyan}{2}}}](21)(29)
    \Edge[color=green,lw=0.5pt,label={{\fontsize{8pt}{8pt}\selectfont \textcolor{green}{3}}}](37)(45)
    \Edge[color=cyan,lw=0.5pt,label={{\fontsize{8pt}{8pt}\selectfont \textcolor{cyan}{2}}}](45)(53)
    \Edge[color=cyan,lw=0.5pt,label={{\fontsize{8pt}{8pt}\selectfont \textcolor{cyan}{2}}}](53)(61) 
    
    \Edge[color=red,lw=0.5pt,label={{\fontsize{8pt}{8pt}\selectfont \textcolor{red}{1}}}](6)(14)
    \Edge[color=red,lw=0.5pt,label={{\fontsize{8pt}{8pt}\selectfont \textcolor{red}{1}}}](22)(30)
    \Edge[color=red,lw=0.5pt,label={{\fontsize{8pt}{8pt}\selectfont \textcolor{red}{1}}}](38)(46)
    \Edge[color=red,lw=0.5pt,label={{\fontsize{8pt}{8pt}\selectfont \textcolor{red}{1}}}](54)(62) 

    \Edge[color=cyan,lw=0.5pt,label={{\fontsize{8pt}{8pt}\selectfont \textcolor{cyan}{2}}}](7)(15)
    \Edge[color=cyan,lw=0.5pt,label={{\fontsize{8pt}{8pt}\selectfont \textcolor{cyan}{2}}}](23)(31)
    \Edge[color=cyan,lw=0.5pt,label={{\fontsize{8pt}{8pt}\selectfont \textcolor{cyan}{2}}}](39)(47)
    \Edge[color=cyan,lw=0.5pt,label={{\fontsize{8pt}{8pt}\selectfont \textcolor{cyan}{2}}}](55)(63) 

    \Edge[color=red,lw=0.5pt,label={{\fontsize{8pt}{8pt}\selectfont \textcolor{red}{1}}}](8)(16)  
    \Edge[color=red,lw=0.5pt,label={{\fontsize{8pt}{8pt}\selectfont \textcolor{red}{1}}}](16)(24)
    \Edge[color=green,lw=0.5pt,label={{\fontsize{8pt}{8pt}\selectfont \textcolor{green}{3}}}](24)(32)
    \Edge[color=green,lw=0.5pt,label={{\fontsize{8pt}{8pt}\selectfont \textcolor{green}{3}}}](32)(40)
    \Edge[color=red,lw=0.5pt,label={{\fontsize{8pt}{8pt}\selectfont \textcolor{red}{1}}}](40)(48)
    \Edge[color=red,lw=0.5pt,label={{\fontsize{8pt}{8pt}\selectfont \textcolor{red}{1}}}](48)(56)
    \Edge[color=green,lw=0.5pt,label={{\fontsize{8pt}{8pt}\selectfont \textcolor{green}{3}}}](56)(64)    
\end{tikzpicture}
\caption{A $\chi_i'$-coloring of $S_{C_4}^3$ with $3$ colors, which leads to the formula $\chi_i'(S_{C_4}^n)=3$ for all $n\ge 2$. The  extreme vertices of $S_3^n$ are shaded.}
\label{fig:SC_4}
\end{center}
\end{figure}

\begin{figure}[H]
    \centering
    \begin{tikzpicture}[scale=1,inner sep=1pt]

\def\r{0.6}
\def\R{1.8}

\newcommand{\smallpentagon}[4]{
    \foreach \i/\angle in {1/90,2/18,3/-54,4/-126,5/162}{
        \pgfmathtruncatemacro{\labelnum}{#4 + \i}
        \node[circle,draw,inner sep=1.5pt] (#3\labelnum) at ({#1+\r*cos(\angle)}, {#2+\r*sin(\angle)}) {};
    }
}

\newcommand{\bigpentagon}[3]{
    \smallpentagon{#1+\R*cos(90)}{#2+\R*sin(90)}{#3}{0}     
    \smallpentagon{#1+\R*cos(18)}{#2+\R*sin(18)}{#3}{5}     
    \smallpentagon{#1+\R*cos(-54)}{#2+\R*sin(-54)}{#3}{10}  
    \smallpentagon{#1+\R*cos(-126)}{#2+\R*sin(-126)}{#3}{15}
    \smallpentagon{#1+\R*cos(162)}{#2+\R*sin(162)}{#3}{20}  
}

\def\RR{6} 
\bigpentagon{\RR*cos(90)}{\RR*sin(90)}{a}
\bigpentagon{\RR*cos(18)}{\RR*sin(18)}{b}
\bigpentagon{\RR*cos(-54)}{\RR*sin(-54)}{c}
\bigpentagon{\RR*cos(-126)}{\RR*sin(-126)}{d}
\bigpentagon{\RR*cos(162)}{\RR*sin(162)}{v}

\Edge[color=cyan,lw=0.5pt,label={{\fontsize{8pt}{8pt}\selectfont \textcolor{cyan}{2}}}](a1)(a2)
\Edge[color=cyan,lw=0.5pt,label={{\fontsize{8pt}{8pt}\selectfont \textcolor{cyan}{2}}}](a2)(a6)
\Edge[color=cyan,lw=0.5pt,label={{\fontsize{8pt}{8pt}\selectfont \textcolor{cyan}{2}}}](a2)(a3)
\Edge[color=cyan,lw=0.5pt,label={{\fontsize{8pt}{8pt}\selectfont \textcolor{cyan}{2}}}](a12)(a13)
\Edge[color=cyan,lw=0.5pt,label={{\fontsize{8pt}{8pt}\selectfont \textcolor{cyan}{2}}}](a13)(a14)
\Edge[color=cyan,lw=0.5pt,label={{\fontsize{8pt}{8pt}\selectfont \textcolor{cyan}{2}}}](a25)(a21)
\Edge[color=cyan,lw=0.5pt,label={{\fontsize{8pt}{8pt}\selectfont \textcolor{cyan}{2}}}](a25)(a24)
\Edge[color=cyan,lw=0.5pt,label={{\fontsize{8pt}{8pt}\selectfont \textcolor{cyan}{2}}}](a25)(v1)
\Edge[color=cyan,lw=0.5pt,label={{\fontsize{8pt}{8pt}\selectfont \textcolor{cyan}{2}}}](v6)(v7)
\Edge[color=cyan,lw=0.5pt,label={{\fontsize{8pt}{8pt}\selectfont \textcolor{cyan}{2}}}](v21)(v25)
\Edge[color=cyan,lw=0.5pt,label={{\fontsize{8pt}{8pt}\selectfont \textcolor{cyan}{2}}}](v25)(v24)
\Edge[color=cyan,lw=0.5pt,label={{\fontsize{8pt}{8pt}\selectfont \textcolor{cyan}{2}}}](v25)(v24)
\Edge[color=cyan,lw=0.5pt,label={{\fontsize{8pt}{8pt}\selectfont \textcolor{cyan}{2}}}](v14)(v13)
\Edge[color=cyan,lw=0.5pt,label={{\fontsize{8pt}{8pt}\selectfont \textcolor{cyan}{2}}}](v16)(v17)
\Edge[color=cyan,lw=0.5pt,label={{\fontsize{8pt}{8pt}\selectfont \textcolor{cyan}{2}}}](v19)(d25)
\Edge[color=cyan,lw=0.5pt,label={{\fontsize{8pt}{8pt}\selectfont \textcolor{cyan}{2}}}](d21)(d25)
\Edge[color=cyan,lw=0.5pt,label={{\fontsize{8pt}{8pt}\selectfont \textcolor{cyan}{2}}}](d24)(d25)
\Edge[color=cyan,lw=0.5pt,label={{\fontsize{8pt}{8pt}\selectfont \textcolor{cyan}{2}}}](d19)(d18)
\Edge[color=cyan,lw=0.5pt,label={{\fontsize{8pt}{8pt}\selectfont \textcolor{cyan}{2}}}](d18)(d14)
\Edge[color=cyan,lw=0.5pt,label={{\fontsize{8pt}{8pt}\selectfont \textcolor{cyan}{2}}}](d18)(d17)
\Edge[color=cyan,lw=0.5pt,label={{\fontsize{8pt}{8pt}\selectfont \textcolor{cyan}{2}}}](d11)(d12)
\Edge[color=cyan,lw=0.5pt,label={{\fontsize{8pt}{8pt}\selectfont \textcolor{cyan}{2}}}](d12)(d8)
\Edge[color=cyan,lw=0.5pt,label={{\fontsize{8pt}{8pt}\selectfont \textcolor{cyan}{2}}}](d1)(d2)
\Edge[color=cyan,lw=0.5pt,label={{\fontsize{8pt}{8pt}\selectfont \textcolor{cyan}{2}}}](d2)(d6)
\Edge[color=cyan,lw=0.5pt,label={{\fontsize{8pt}{8pt}\selectfont \textcolor{cyan}{2}}}](c11)(c15)
\Edge[color=cyan,lw=0.5pt,label={{\fontsize{8pt}{8pt}\selectfont \textcolor{cyan}{2}}}](c20)(c16)
\Edge[color=cyan,lw=0.5pt,label={{\fontsize{8pt}{8pt}\selectfont \textcolor{cyan}{2}}}](c16)(c17)
\Edge[color=cyan,lw=0.5pt,label={{\fontsize{8pt}{8pt}\selectfont \textcolor{cyan}{2}}}](c25)(c21)
\Edge[color=cyan,lw=0.5pt,label={{\fontsize{8pt}{8pt}\selectfont \textcolor{cyan}{2}}}](c21)(c5)
\Edge[color=cyan,lw=0.5pt,label={{\fontsize{8pt}{8pt}\selectfont \textcolor{cyan}{2}}}](c6)(c7)
\Edge[color=cyan,lw=0.5pt,label={{\fontsize{8pt}{8pt}\selectfont \textcolor{cyan}{2}}}](c8)(c7)
\Edge[color=cyan,lw=0.5pt,label={{\fontsize{8pt}{8pt}\selectfont \textcolor{cyan}{2}}}](c21)(c22)
\Edge[color=cyan,lw=0.5pt,label={{\fontsize{8pt}{8pt}\selectfont \textcolor{cyan}{2}}}](b19)(b20)
\Edge[color=cyan,lw=0.5pt,label={{\fontsize{8pt}{8pt}\selectfont \textcolor{cyan}{2}}}](b16)(b20)
\Edge[color=cyan,lw=0.5pt,label={{\fontsize{8pt}{8pt}\selectfont \textcolor{cyan}{2}}}](b19)(b20)
\Edge[color=cyan,lw=0.5pt,label={{\fontsize{8pt}{8pt}\selectfont \textcolor{cyan}{2}}}](b24)(b20)
\Edge[color=cyan,lw=0.5pt,label={{\fontsize{8pt}{8pt}\selectfont \textcolor{cyan}{2}}}](b1)(b2)
\Edge[color=cyan,lw=0.5pt,label={{\fontsize{8pt}{8pt}\selectfont \textcolor{cyan}{2}}}](b6)(b2)
\Edge[color=cyan,lw=0.5pt,label={{\fontsize{8pt}{8pt}\selectfont \textcolor{cyan}{2}}}](b3)(b2)
\Edge[color=cyan,lw=0.5pt,label={{\fontsize{8pt}{8pt}\selectfont \textcolor{cyan}{2}}}](b9)(b8)
\Edge[color=cyan,lw=0.5pt,label={{\fontsize{8pt}{8pt}\selectfont \textcolor{cyan}{2}}}](b8)(b12)

\Edge[color=red,lw=0.5pt,label={{\fontsize{8pt}{8pt}\selectfont \textcolor{red}{1}}}](a1)(a5)
\Edge[color=red,lw=0.5pt,label={{\fontsize{8pt}{8pt}\selectfont \textcolor{red}{1}}}](a21)(a5)
\Edge[color=red,lw=0.5pt,label={{\fontsize{8pt}{8pt}\selectfont \textcolor{red}{1}}}](a4)(a5)
\Edge[color=red,lw=0.5pt,label={{\fontsize{8pt}{8pt}\selectfont \textcolor{red}{1}}}](a16)(a17)
\Edge[color=red,lw=0.5pt,label={{\fontsize{8pt}{8pt}\selectfont \textcolor{red}{1}}}](a17)(a18)
\Edge[color=red,lw=0.5pt,label={{\fontsize{8pt}{8pt}\selectfont \textcolor{red}{1}}}](a11)(a15)
\Edge[color=red,lw=0.5pt,label={{\fontsize{8pt}{8pt}\selectfont \textcolor{red}{1}}}](a9)(a10)
\Edge[color=red,lw=0.5pt,label={{\fontsize{8pt}{8pt}\selectfont \textcolor{red}{1}}}](a10)(a6)
\Edge[color=red,lw=0.5pt,label={{\fontsize{8pt}{8pt}\selectfont \textcolor{red}{1}}}](a11)(a15)
\Edge[color=red,lw=0.5pt,label={{\fontsize{8pt}{8pt}\selectfont \textcolor{red}{1}}}](a11)(a12)
\Edge[color=red,lw=0.5pt,label={{\fontsize{8pt}{8pt}\selectfont \textcolor{red}{1}}}](b1)(b5)
\Edge[color=red,lw=0.5pt,label={{\fontsize{8pt}{8pt}\selectfont \textcolor{red}{1}}}](b5)(b21)
\Edge[color=red,lw=0.5pt,label={{\fontsize{8pt}{8pt}\selectfont \textcolor{red}{1}}}](b5)(b4)
\Edge[color=red,lw=0.5pt,label={{\fontsize{8pt}{8pt}\selectfont \textcolor{red}{1}}}](b24)(b23)
\Edge[color=red,lw=0.5pt,label={{\fontsize{8pt}{8pt}\selectfont \textcolor{red}{1}}}](b19)(b18)
\Edge[color=red,lw=0.5pt,label={{\fontsize{8pt}{8pt}\selectfont \textcolor{red}{1}}}](b17)(b18)
\Edge[color=red,lw=0.5pt,label={{\fontsize{8pt}{8pt}\selectfont \textcolor{red}{1}}}](b11)(b12)
\Edge[color=red,lw=0.5pt,label={{\fontsize{8pt}{8pt}\selectfont \textcolor{red}{1}}}](b13)(b12)
\Edge[color=red,lw=0.5pt,label={{\fontsize{8pt}{8pt}\selectfont \textcolor{red}{1}}}](b9)(b10)
\Edge[color=red,lw=0.5pt,label={{\fontsize{8pt}{8pt}\selectfont \textcolor{red}{1}}}](b10)(b6)
\Edge[color=red,lw=0.5pt,label={{\fontsize{8pt}{8pt}\selectfont \textcolor{red}{1}}}](c5)(c4)
\Edge[color=red,lw=0.5pt,label={{\fontsize{8pt}{8pt}\selectfont \textcolor{red}{1}}}](c2)(c6)
\Edge[color=red,lw=0.5pt,label={{\fontsize{8pt}{8pt}\selectfont \textcolor{red}{1}}}](c6)(c10)
\Edge[color=red,lw=0.5pt,label={{\fontsize{8pt}{8pt}\selectfont \textcolor{red}{1}}}](c8)(c12)
\Edge[color=red,lw=0.5pt,label={{\fontsize{8pt}{8pt}\selectfont \textcolor{red}{1}}}](c12)(c11)
\Edge[color=red,lw=0.5pt,label={{\fontsize{8pt}{8pt}\selectfont \textcolor{red}{1}}}](c12)(c13)
\Edge[color=red,lw=0.5pt,label={{\fontsize{8pt}{8pt}\selectfont \textcolor{red}{1}}}](c17)(c18)
\Edge[color=red,lw=0.5pt,label={{\fontsize{8pt}{8pt}\selectfont \textcolor{red}{1}}}](c18)(c19)
\Edge[color=red,lw=0.5pt,label={{\fontsize{8pt}{8pt}\selectfont \textcolor{red}{1}}}](c22)(c23)
\Edge[color=red,lw=0.5pt,label={{\fontsize{8pt}{8pt}\selectfont \textcolor{red}{1}}}](v21)(v5)
\Edge[color=red,lw=0.5pt,label={{\fontsize{8pt}{8pt}\selectfont \textcolor{red}{1}}}](v5)(v1)
\Edge[color=red,lw=0.5pt,label={{\fontsize{8pt}{8pt}\selectfont \textcolor{red}{1}}}](v4)(v5)
\Edge[color=red,lw=0.5pt,label={{\fontsize{8pt}{8pt}\selectfont \textcolor{red}{1}}}](v24)(v23)
\Edge[color=red,lw=0.5pt,label={{\fontsize{8pt}{8pt}\selectfont \textcolor{red}{1}}}](v20)(v24)
\Edge[color=red,lw=0.5pt,label={{\fontsize{8pt}{8pt}\selectfont \textcolor{red}{1}}}](v14)(v15)
\Edge[color=red,lw=0.5pt,label={{\fontsize{8pt}{8pt}\selectfont \textcolor{red}{1}}}](v15)(v11)
\Edge[color=red,lw=0.5pt,label={{\fontsize{8pt}{8pt}\selectfont \textcolor{red}{1}}}](v9)(v8)
\Edge[color=red,lw=0.5pt,label={{\fontsize{8pt}{8pt}\selectfont \textcolor{red}{1}}}](v7)(v8)
\Edge[color=red,lw=0.5pt,label={{\fontsize{8pt}{8pt}\selectfont \textcolor{red}{1}}}](d24)(d20)
\Edge[color=red,lw=0.5pt,label={{\fontsize{8pt}{8pt}\selectfont \textcolor{red}{1}}}](d20)(d16)
\Edge[color=red,lw=0.5pt,label={{\fontsize{8pt}{8pt}\selectfont \textcolor{red}{1}}}](d20)(d19)
\Edge[color=red,lw=0.5pt,label={{\fontsize{8pt}{8pt}\selectfont \textcolor{red}{1}}}](d14)(d15)
\Edge[color=red,lw=0.5pt,label={{\fontsize{8pt}{8pt}\selectfont \textcolor{red}{1}}}](d9)(d8)
\Edge[color=red,lw=0.5pt,label={{\fontsize{8pt}{8pt}\selectfont \textcolor{red}{1}}}](d8)(d7)
\Edge[color=red,lw=0.5pt,label={{\fontsize{8pt}{8pt}\selectfont \textcolor{red}{1}}}](d4)(d3)
\Edge[color=red,lw=0.5pt,label={{\fontsize{8pt}{8pt}\selectfont \textcolor{red}{1}}}](d3)(d2)
\Edge[color=red,lw=0.5pt,label={{\fontsize{8pt}{8pt}\selectfont \textcolor{red}{1}}}](a24)(a23)

\Edge[color=green,lw=0.5pt,label={{\fontsize{8pt}{8pt}\selectfont \textcolor{green}{3}}}](a7)(b1)
\Edge[color=green,lw=0.5pt,label={{\fontsize{8pt}{8pt}\selectfont \textcolor{green}{3}}}](c1)(c2)
\Edge[color=green,lw=0.5pt,label={{\fontsize{8pt}{8pt}\selectfont \textcolor{green}{3}}}](c1)(c5)
\Edge[color=green,lw=0.5pt,label={{\fontsize{8pt}{8pt}\selectfont \textcolor{green}{3}}}](c10)(c9)
\Edge[color=green,lw=0.5pt,label={{\fontsize{8pt}{8pt}\selectfont \textcolor{green}{3}}}](c9)(c8)
\Edge[color=green,lw=0.5pt,label={{\fontsize{8pt}{8pt}\selectfont \textcolor{green}{3}}}](c15)(c14)
\Edge[color=green,lw=0.5pt,label={{\fontsize{8pt}{8pt}\selectfont \textcolor{green}{3}}}](c14)(c13)
\Edge[color=green,lw=0.5pt,label={{\fontsize{8pt}{8pt}\selectfont \textcolor{green}{3}}}](c24)(c23)
\Edge[color=green,lw=0.5pt,label={{\fontsize{8pt}{8pt}\selectfont \textcolor{green}{3}}}](c24)(c25)
\Edge[color=green,lw=0.5pt,label={{\fontsize{8pt}{8pt}\selectfont \textcolor{green}{3}}}](c24)(c20)
\Edge[color=green,lw=0.5pt,label={{\fontsize{8pt}{8pt}\selectfont \textcolor{green}{3}}}](c18)(c14)
\Edge[color=green,lw=0.5pt,label={{\fontsize{8pt}{8pt}\selectfont \textcolor{green}{3}}}](d1)(d5)
\Edge[color=green,lw=0.5pt,label={{\fontsize{8pt}{8pt}\selectfont \textcolor{green}{3}}}](d4)(d5)
\Edge[color=green,lw=0.5pt,label={{\fontsize{8pt}{8pt}\selectfont \textcolor{green}{3}}}](d22)(d23)
\Edge[color=green,lw=0.5pt,label={{\fontsize{8pt}{8pt}\selectfont \textcolor{green}{3}}}](d23)(d24)
\Edge[color=green,lw=0.5pt,label={{\fontsize{8pt}{8pt}\selectfont \textcolor{green}{3}}}](d16)(d17)
\Edge[color=green,lw=0.5pt,label={{\fontsize{8pt}{8pt}\selectfont \textcolor{green}{3}}}](d14)(d13)
\Edge[color=green,lw=0.5pt,label={{\fontsize{8pt}{8pt}\selectfont \textcolor{green}{3}}}](d13)(d12)
\Edge[color=green,lw=0.5pt,label={{\fontsize{8pt}{8pt}\selectfont \textcolor{green}{3}}}](d10)(d9)
\Edge[color=green,lw=0.5pt,label={{\fontsize{8pt}{8pt}\selectfont \textcolor{green}{3}}}](v20)(v19)
\Edge[color=green,lw=0.5pt,label={{\fontsize{8pt}{8pt}\selectfont \textcolor{green}{3}}}](v20)(v16)
\Edge[color=green,lw=0.5pt,label={{\fontsize{8pt}{8pt}\selectfont \textcolor{green}{3}}}](v13)(v12)
\Edge[color=green,lw=0.5pt,label={{\fontsize{8pt}{8pt}\selectfont \textcolor{green}{3}}}](v12)(v11)
\Edge[color=green,lw=0.5pt,label={{\fontsize{8pt}{8pt}\selectfont \textcolor{green}{3}}}](v12)(v8)
\Edge[color=green,lw=0.5pt,label={{\fontsize{8pt}{8pt}\selectfont \textcolor{green}{3}}}](v2)(v1)
\Edge[color=green,lw=0.5pt,label={{\fontsize{8pt}{8pt}\selectfont \textcolor{green}{3}}}](v2)(v3)
\Edge[color=green,lw=0.5pt,label={{\fontsize{8pt}{8pt}\selectfont \textcolor{green}{3}}}](v2)(v6)
\Edge[color=green,lw=0.5pt,label={{\fontsize{8pt}{8pt}\selectfont \textcolor{green}{3}}}](v21)(v22)
\Edge[color=green,lw=0.5pt,label={{\fontsize{8pt}{8pt}\selectfont \textcolor{green}{3}}}](v22)(v23)
\Edge[color=green,lw=0.5pt,label={{\fontsize{8pt}{8pt}\selectfont \textcolor{green}{3}}}](a21)(a22)
\Edge[color=green,lw=0.5pt,label={{\fontsize{8pt}{8pt}\selectfont \textcolor{green}{3}}}](a22)(a23)
\Edge[color=green,lw=0.5pt,label={{\fontsize{8pt}{8pt}\selectfont \textcolor{green}{3}}}](a18)(a19)
\Edge[color=green,lw=0.5pt,label={{\fontsize{8pt}{8pt}\selectfont \textcolor{green}{3}}}](a14)(a18)
\Edge[color=green,lw=0.5pt,label={{\fontsize{8pt}{8pt}\selectfont \textcolor{green}{3}}}](a6)(a7)
\Edge[color=green,lw=0.5pt,label={{\fontsize{8pt}{8pt}\selectfont \textcolor{green}{3}}}](a7)(a8)
\Edge[color=green,lw=0.5pt,label={{\fontsize{8pt}{8pt}\selectfont \textcolor{green}{3}}}](a4)(a3)
\Edge[color=green,lw=0.5pt,label={{\fontsize{8pt}{8pt}\selectfont \textcolor{green}{3}}}](b3)(b4)
\Edge[color=green,lw=0.5pt,label={{\fontsize{8pt}{8pt}\selectfont \textcolor{green}{3}}}](b6)(b7)
\Edge[color=green,lw=0.5pt,label={{\fontsize{8pt}{8pt}\selectfont \textcolor{green}{3}}}](b7)(b8)
\Edge[color=green,lw=0.5pt,label={{\fontsize{8pt}{8pt}\selectfont \textcolor{green}{3}}}](b18)(b14)
\Edge[color=green,lw=0.5pt,label={{\fontsize{8pt}{8pt}\selectfont \textcolor{green}{3}}}](b14)(b13)
\Edge[color=green,lw=0.5pt,label={{\fontsize{8pt}{8pt}\selectfont \textcolor{green}{3}}}](b14)(b15)
\Edge[color=green,lw=0.5pt,label={{\fontsize{8pt}{8pt}\selectfont \textcolor{green}{3}}}](b21)(b25)
\Edge[color=green,lw=0.5pt,label={{\fontsize{8pt}{8pt}\selectfont \textcolor{green}{3}}}](b25)(b24)

\Edge[color=violet,lw=0.5pt,label={{\fontsize{8pt}{8pt}\selectfont \textcolor{violet}{4}}}](c4)(c3)
\Edge[color=violet,lw=0.5pt,label={{\fontsize{8pt}{8pt}\selectfont \textcolor{violet}{4}}}](c2)(c3)
\Edge[color=violet,lw=0.5pt,label={{\fontsize{8pt}{8pt}\selectfont \textcolor{violet}{4}}}](c20)(c19)
\Edge[color=violet,lw=0.5pt,label={{\fontsize{8pt}{8pt}\selectfont \textcolor{violet}{4}}}](d13)(c19)
\Edge[color=violet,lw=0.5pt,label={{\fontsize{8pt}{8pt}\selectfont \textcolor{violet}{4}}}](b13)(c7)
\Edge[color=violet,lw=0.5pt,label={{\fontsize{8pt}{8pt}\selectfont \textcolor{violet}{4}}}](a24)(a20)
\Edge[color=violet,lw=0.5pt,label={{\fontsize{8pt}{8pt}\selectfont \textcolor{violet}{4}}}](a20)(a16)
\Edge[color=violet,lw=0.5pt,label={{\fontsize{8pt}{8pt}\selectfont \textcolor{violet}{4}}}](a20)(a19)
\Edge[color=violet,lw=0.5pt,label={{\fontsize{8pt}{8pt}\selectfont \textcolor{violet}{4}}}](a14)(a15)
\Edge[color=violet,lw=0.5pt,label={{\fontsize{8pt}{8pt}\selectfont \textcolor{violet}{4}}}](a12)(a8)
\Edge[color=violet,lw=0.5pt,label={{\fontsize{8pt}{8pt}\selectfont \textcolor{violet}{4}}}](a8)(a9)
\Edge[color=violet,lw=0.5pt,label={{\fontsize{8pt}{8pt}\selectfont \textcolor{violet}{4}}}](b21)(b22)
\Edge[color=violet,lw=0.5pt,label={{\fontsize{8pt}{8pt}\selectfont \textcolor{violet}{4}}}](b22)(b23)
\Edge[color=violet,lw=0.5pt,label={{\fontsize{8pt}{8pt}\selectfont \textcolor{violet}{4}}}](b16)(b17)
\Edge[color=violet,lw=0.5pt,label={{\fontsize{8pt}{8pt}\selectfont \textcolor{violet}{4}}}](b15)(b11)
\Edge[color=violet,lw=0.5pt,label={{\fontsize{8pt}{8pt}\selectfont \textcolor{violet}{4}}}](d21)(d22)
\Edge[color=violet,lw=0.5pt,label={{\fontsize{8pt}{8pt}\selectfont \textcolor{violet}{4}}}](d21)(d5)
\Edge[color=violet,lw=0.5pt,label={{\fontsize{8pt}{8pt}\selectfont \textcolor{violet}{4}}}](d11)(d15)
\Edge[color=violet,lw=0.5pt,label={{\fontsize{8pt}{8pt}\selectfont \textcolor{violet}{4}}}](d10)(d6)
\Edge[color=violet,lw=0.5pt,label={{\fontsize{8pt}{8pt}\selectfont \textcolor{violet}{4}}}](d7)(d6)
\Edge[color=violet,lw=0.5pt,label={{\fontsize{8pt}{8pt}\selectfont \textcolor{violet}{4}}}](v4)(v3)
\Edge[color=violet,lw=0.5pt,label={{\fontsize{8pt}{8pt}\selectfont \textcolor{violet}{4}}}](v6)(v10)
\Edge[color=violet,lw=0.5pt,label={{\fontsize{8pt}{8pt}\selectfont \textcolor{violet}{4}}}](v10)(v9)
\Edge[color=violet,lw=0.5pt,label={{\fontsize{8pt}{8pt}\selectfont \textcolor{violet}{4}}}](v18)(v14)
\Edge[color=violet,lw=0.5pt,label={{\fontsize{8pt}{8pt}\selectfont \textcolor{violet}{4}}}](v18)(v19)
\Edge[color=violet,lw=0.5pt,label={{\fontsize{8pt}{8pt}\selectfont \textcolor{violet}{4}}}](v17)(v18)

\end{tikzpicture}
    \caption{A $\chi_i'$-coloring of $S_{C_5}^3$ with $4$ colors.}
    \label{IECofS_C5^3}
\end{figure}
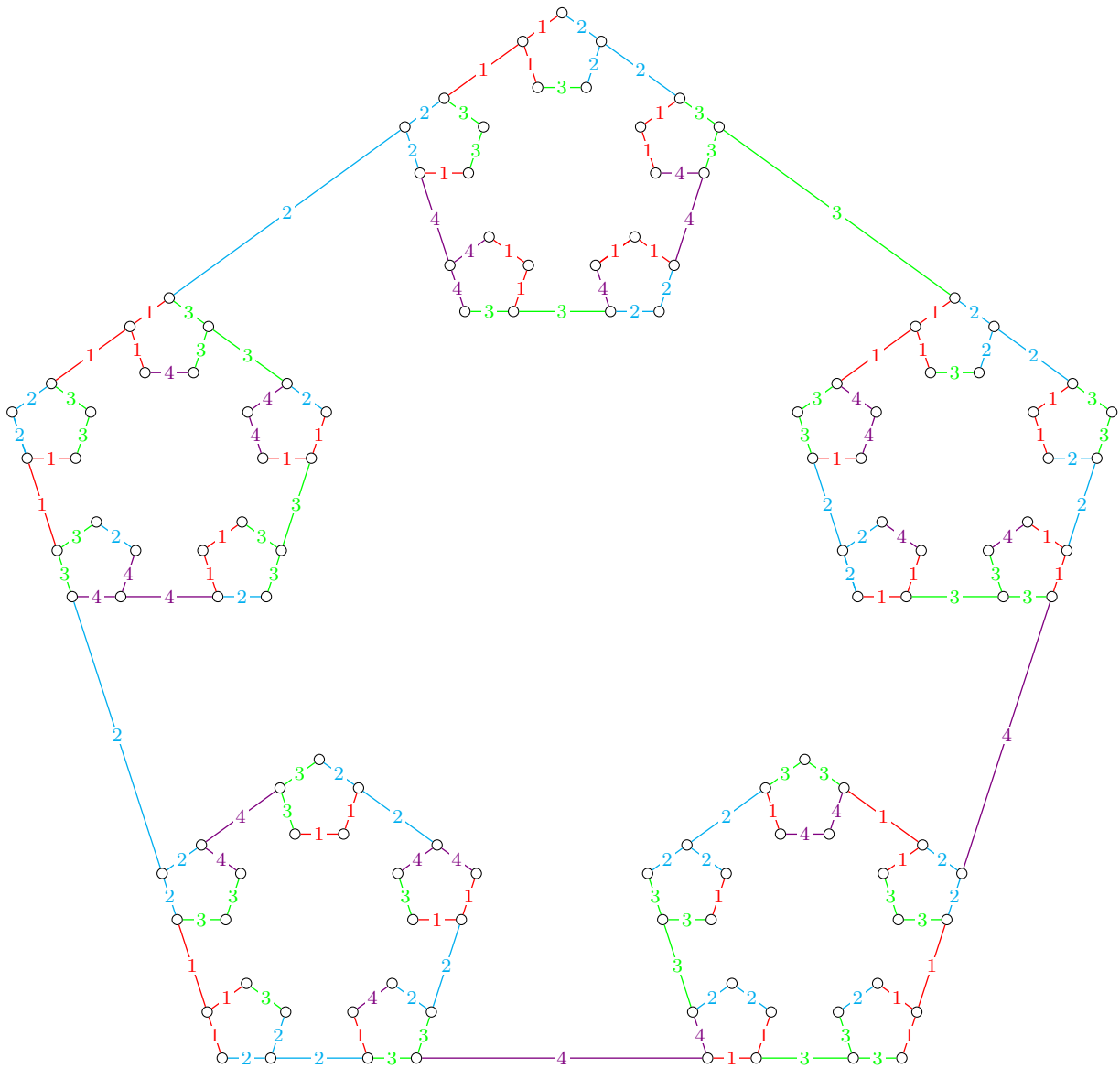

As an application of Theorem~\ref{thm:triangle-free} we determine the injective chromatic indices of the generalized Sierpi\'{n}ski graphs over the cycles $C_4$ and $C_5$. 
First note that $\chi_i'(C_4)=2$. It is easy to see that $\chi_i'(S^2_{C_4})\ge 3$, while from Figure~\ref{fig:SC_4} we can infer that indeed $3$ colors suffice for an injective edge coloring of $S^2_{C_4}$. In addition, the same figure presents an injective edge coloring of $S^3_{C_4}$ using $3$ colors.  Note that the coloring in Figure~\ref{fig:SC_4} is a $\chi_i'$-coloring of $S_{C_4}^3$, which satisfies the conditions in item {\it (ii)} of Theorem~\ref{thm:SG3}. We infer the following result.

\begin{corollary}
For $n\in \mathbb{N}$, we have
$$\chi_i'(S_{C_4}^n)= \begin{cases}
	           2, ~~\textrm{if}~ n=1,\\
	           3, ~~\textrm{if}~ n\ge 2.\\
	          \end{cases}$$
\end{corollary}

Now, consider the generalized Sierpi\'{n}ski graph over $C_5$. Clearly, $\chi_i'(C_5)=3$. With some case analysis one can prove that $\chi_i(S_{C_5}^2)\ge 4$, while Figure~\ref{IECofS_C5^3} shows that $\chi_i(S_{C_5}^2)=4$. The same figure presents a $\chi_i'$-coloring of $S_{C_5}^3$ using $4$ colors. In addition, note that every vertex, which is incident with an edge with color $4$ in that coloring, is at distance at least $3$ from all extreme vertices of $S_{C_5}^3$. Therefore, by item \textit{(i)} in Theorem~\ref{thm:triangle-free}, $\chi_i'(S_{C_5}^n)=\chi_i(S_{C_5}^3)=4$. We infer the following result.  
\begin{corollary}
For $n\in \mathbb{N}$, we have
$$\chi_i'(S_{C_5}^n)= \begin{cases}
	           3, ~~\textrm{if}~ n=1,\\
	           4, ~~\textrm{if}~ n\ge 2.\\
	          \end{cases}$$
\end{corollary}

In a similar way as for the base graph $C_4$, one can find a construction when the base graph is $C_6$. See Figure~\ref{IECofS_C6^2}, which represents an injective edge coloring of $S_{C_6}^2$ with $3$ colors. Using the same coloring for each copy of $S_{C_6}^2$ in $S_{C_6}^3$ and the edges connecting the extreme vertices with color $1$ we obtain a $\chi_i'$-coloring of $S_{C_6}^3$ with $3$ colors satisfying the properties from Theorem~\ref{thm:triangle-free}{\it (ii)}. We derive that $\chi_i'(S_{C_6}^n)=3$ for all $n\in \mathbb{N}$.  

\begin{figure}[H]
    \centering
    \begin{tikzpicture}[scale=1,inner sep=1.4pt]

\def\r{0.8}
\def\R{3}

\newcommand{\hexagonsix}[3]{
  \foreach \i/\angle in {1/90,2/30,3/-30,4/-90,5/-150,6/150} {
    \node[circle,draw,inner sep=1.7pt] (#3\i) at ({#1+\r*cos(\angle)}, {#2+\r*sin(\angle)}) {};
  }
}

\hexagonsix{\R*cos(90)}{\R*sin(90)}{a}
\hexagonsix{\R*cos(30)}{\R*sin(30)}{b}
\hexagonsix{\R*cos(-30)}{\R*sin(-30)}{c}
\hexagonsix{\R*cos(-90)}{\R*sin(-90)}{d}
\hexagonsix{\R*cos(-150)}{\R*sin(-150)}{u}
\hexagonsix{\R*cos(150)}{\R*sin(150)}{v}

    \Edge[color=red,lw=0.5pt,label={{\fontsize{6pt}{8pt}\selectfont \textcolor{red}{1}}}](a1)(a2)
    \Edge[color=red,lw=0.5pt,label={{\fontsize{6pt}{8pt}\selectfont \textcolor{red}{1}}}](a1)(a6)
    \Edge[color=cyan,lw=0.5pt,label={{\fontsize{6pt}{8pt}\selectfont \textcolor{cyan}{2}}}](a6)(a5)
    \Edge[color=cyan,lw=0.5pt,label={{\fontsize{6pt}{8pt}\selectfont \textcolor{cyan}{2}}}](a5)(a4)
    \Edge[color=green,lw=0.5pt,label={{\fontsize{6pt}{8pt}\selectfont \textcolor{green}{3}}}](a3)(a4)
    \Edge[color=green,lw=0.5pt,label={{\fontsize{6pt}{8pt}\selectfont \textcolor{green}{3}}}](a3)(a2)
    \Edge[color=cyan,lw=0.5pt,label={{\fontsize{6pt}{8pt}\selectfont \textcolor{cyan}{2}}}](b1)(a2)
    \Edge[color=cyan,lw=0.5pt,label={{\fontsize{6pt}{8pt}\selectfont \textcolor{cyan}{2}}}](b1)(b2)
    \Edge[color=cyan,lw=0.5pt,label={{\fontsize{6pt}{8pt}\selectfont \textcolor{cyan}{2}}}](b1)(b6)
    \Edge[color=red,lw=0.5pt,label={{\fontsize{6pt}{8pt}\selectfont \textcolor{red}{1}}}](b6)(b5)
    \Edge[color=red,lw=0.5pt,label={{\fontsize{6pt}{8pt}\selectfont \textcolor{red}{1}}}](b5)(b4)
    \Edge[color=green,lw=0.5pt,label={{\fontsize{6pt}{8pt}\selectfont \textcolor{green}{3}}}](b3)(b4)
    \Edge[color=green,lw=0.5pt,label={{\fontsize{6pt}{8pt}\selectfont \textcolor{green}{3}}}](b3)(b2)
    \Edge[color=red,lw=0.5pt,label={{\fontsize{6pt}{8pt}\selectfont \textcolor{red}{1}}}](a1)(a2) 
    \Edge[color=green,lw=0.5pt,label={{\fontsize{6pt}{8pt}\selectfont \textcolor{green}{3}}}](b3)(c2)
    \Edge[color=red,lw=0.5pt,label={{\fontsize{6pt}{8pt}\selectfont \textcolor{red}{1}}}](c3)(c2)
    \Edge[color=cyan,lw=0.5pt,label={{\fontsize{6pt}{8pt}\selectfont \textcolor{cyan}{2}}}](c1)(c2)
    \Edge[color=red,lw=0.5pt,label={{\fontsize{6pt}{8pt}\selectfont \textcolor{red}{1}}}](c3)(c4)
    \Edge[color=green,lw=0.5pt,label={{\fontsize{6pt}{8pt}\selectfont \textcolor{green}{3}}}](c5)(c4)
    \Edge[color=green,lw=0.5pt,label={{\fontsize{6pt}{8pt}\selectfont \textcolor{green}{3}}}](c5)(c6)
    \Edge[color=cyan,lw=0.5pt,label={{\fontsize{6pt}{8pt}\selectfont \textcolor{cyan}{2}}}](c6)(c1)
    \Edge[color=cyan,lw=0.5pt,label={{\fontsize{6pt}{8pt}\selectfont \textcolor{cyan}{2}}}](c4)(d3)
    \Edge[color=cyan,lw=0.5pt,label={{\fontsize{6pt}{8pt}\selectfont \textcolor{cyan}{2}}}](d3)(d2)
    \Edge[color=cyan,lw=0.5pt,label={{\fontsize{6pt}{8pt}\selectfont \textcolor{cyan}{2}}}](d3)(d4)
    \Edge[color=green,lw=0.5pt,label={{\fontsize{6pt}{8pt}\selectfont \textcolor{green}{3}}}](d4)(d5)
    \Edge[color=green,lw=0.5pt,label={{\fontsize{6pt}{8pt}\selectfont \textcolor{green}{3}}}](d5)(d6)   
    \Edge[color=red,lw=0.5pt,label={{\fontsize{6pt}{8pt}\selectfont \textcolor{red}{1}}}](d6)(d1)
    \Edge[color=red,lw=0.5pt,label={{\fontsize{6pt}{8pt}\selectfont \textcolor{red}{1}}}](d1)(d2)
    \Edge[color=green,lw=0.5pt,label={{\fontsize{6pt}{8pt}\selectfont \textcolor{green}{3}}}](d5)(u4)  
    \Edge[color=cyan,lw=0.5pt,label={{\fontsize{6pt}{8pt}\selectfont \textcolor{cyan}{2}}}](u4)(u3)
    \Edge[color=cyan,lw=0.5pt,label={{\fontsize{6pt}{8pt}\selectfont \textcolor{cyan}{2}}}](u2)(u3)
    \Edge[color=green,lw=0.5pt,label={{\fontsize{6pt}{8pt}\selectfont \textcolor{green}{3}}}](u2)(u1)
    \Edge[color=green,lw=0.5pt,label={{\fontsize{6pt}{8pt}\selectfont \textcolor{green}{3}}}](u1)(u6) 
    \Edge[color=red,lw=0.5pt,label={{\fontsize{6pt}{8pt}\selectfont \textcolor{red}{1}}}](u4)(u5)
    \Edge[color=red,lw=0.5pt,label={{\fontsize{6pt}{8pt}\selectfont \textcolor{red}{1}}}](u5)(u6)
    \Edge[color=cyan,lw=0.5pt,label={{\fontsize{6pt}{8pt}\selectfont \textcolor{cyan}{2}}}](u6)(v5)  
    \Edge[color=cyan,lw=0.5pt,label={{\fontsize{6pt}{8pt}\selectfont \textcolor{cyan}{2}}}](v5)(v4)
    \Edge[color=cyan,lw=0.5pt,label={{\fontsize{6pt}{8pt}\selectfont \textcolor{cyan}{2}}}](v5)(v6)
    \Edge[color=green,lw=0.5pt,label={{\fontsize{6pt}{8pt}\selectfont \textcolor{green}{3}}}](v6)(v1)
    \Edge[color=green,lw=0.5pt,label={{\fontsize{6pt}{8pt}\selectfont \textcolor{green}{3}}}](v1)(v2) 
    \Edge[color=red,lw=0.5pt,label={{\fontsize{6pt}{8pt}\selectfont \textcolor{red}{1}}}](v2)(v3)
    \Edge[color=red,lw=0.5pt,label={{\fontsize{6pt}{8pt}\selectfont \textcolor{red}{1}}}](v3)(v4)   
    \Edge[color=green,lw=0.5pt,label={{\fontsize{6pt}{8pt}\selectfont \textcolor{green}{3}}}](v1)(a6) 
\end{tikzpicture}
    \caption{A $\chi_i'$-coloring of $S_{C_6}^2$ with 3 colors.}
    \label{IECofS_C6^2}
\end{figure}
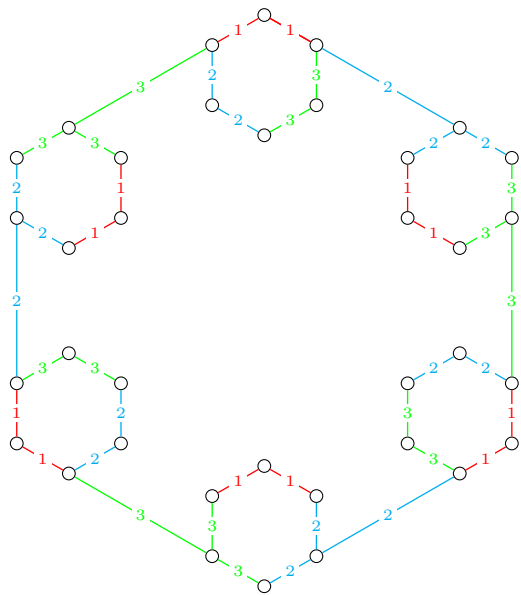

We believe that for Sierpi\'{n}ski graphs over cycles and many other triangle-free graphs as base graphs, Theorem~\ref{thm:triangle-free} can be used for determining their injective chromatic indices. Nevertheless, for graphs $G$ with (many) triangles obtaining a useful general result on $\chi_i'(S_G^n)$ seems more challenging. 

\section*{Acknowledgments}
B.B. acknowledges the financial support of the Slovenian Research and Innovation Agency (research core funding No.\ P1-0297, and projects N1-0285, N1-0431).

\section*{Statements and declarations}

\medskip

{\bf Funding}

\medskip

\noindent B.B.\ acknowledges the support from the Slovenian Research and Innovation Agency (ARIS) under the grants P1-0297, N1-0285, and N1-0431. 

\bigskip

\noindent {\bf Competing interests}

\medskip

\noindent The authors declare that they have no relevant competing financial or non-financial interests to disclose.

\bigskip

\noindent {\bf Author contributions}

\medskip

\noindent All authors were involved at all stages of the preparation of this work and equally share their contribution to the paper. 

\bigskip

\noindent {\bf Data availability}

\medskip

\noindent There is no associated data to this work.

\end{document}